\documentclass[12pt]{amsart}
\usepackage{amscd,verbatim}
\usepackage[all]{xy}    
\usepackage{graphicx,fullpage}
\usepackage{amssymb}
\usepackage{mathrsfs}
\usepackage{multirow}
\usepackage{float}
\usepackage{tikz-cd}
\usepackage{adjustbox}
\usepackage{amsthm}
\usepackage{accents}
\usepackage{mathtools}
\usepackage{mathptmx,enumitem,cite,array}
\usepackage{tikz}
\usetikzlibrary{matrix,arrows}
\usepackage{amsmath}

\usepackage[numbers,sort&compress]{natbib} 
\usepackage[bookmarksnumbered, bookmarksopen,
colorlinks,citecolor=blue,linkcolor=blue,backref]{hyperref}

\usepackage[utf8]{inputenc}
\usepackage{chngcntr}
\usepackage{apptools}
\AtAppendix{\counterwithin{theorem}{section}}
\usepackage[toc,page]{appendix}
\usepackage[all]{xy}

\newtheorem{theorem}{Theorem}[section]
\newtheorem{lemma}[theorem]{Lemma}

\newtheorem{proposition}[theorem]{Proposition}

\newtheorem{definition}[theorem]{Definition}

\theoremstyle{example}
\newtheorem{remark}[theorem]{Remark}

\numberwithin{equation}{section}

\newcommand{\beq}{\begin{equation}}
\newcommand{\eeq}{\end{equation}}

\newcommand{\ZZ}{\mathbb{Z}}
\newcommand{\RR}{\mathbb{R}}
\newcommand{\CC}{\mathbb{C}}

\DeclareMathOperator{\ext}{ext}
\DeclareMathOperator{\rest}{rest}
\DeclareMathOperator{\Fred}{Fred}
\DeclareMathOperator{\im}{im}

\DeclareMathOperator{\id}{id}

\newcommand{\Z}{\mathbb{Z}}

\newcommand{\Ch}{\mathrm{Ch}}

\newcommand{\cM}{\mathcal{M}}
\newcommand{\cO}{\mathcal{O}}

\newcommand{\cK}{\mathcal{K}}

\newcommand{\FF}{\mathcal{F}}

\newcommand{\cA}{\mathcal A}
\newcommand{\Aut}{\rm Aut}

\newcommand {\be}{\begin{equation}}
\newcommand {\ee}{\end{equation}}
\newcommand{\e}{\end{eqnarray*}}

\begin{document}


\title
{On the Chern character in Higher Twisted K-theory\\ and spherical T-duality}

\author{Lachlan Macdonald}
\address{School of Mathematical Sciences,
University of Adelaide, Adelaide 5005, Australia}
\email{lachlan.macdonald@adelaide.edu.au}

 \author{Varghese Mathai}
\address{School of Mathematical Sciences,
University of Adelaide, Adelaide 5005, Australia}
\email{mathai.varghese@adelaide.edu.au}

\author{Hemanth Saratchandran}
\address{School of Mathematical Sciences,
University of Adelaide, Adelaide 5005, Australia}
\email{hemanth.saratchandran@adelaide.edu.au}

\date{}

\begin{abstract}
In this paper, we construct for higher twists that arise from cohomotopy classes, the Chern character in higher twisted K-theory,
that maps into higher twisted cohomology. We show that it gives rise to an isomorphism between higher twisted K-theory 
and higher twisted cohomology over the reals.
Finally we compute spherical T-duality in higher twisted K-theory and higher twisted cohomology in very general cases.
\end{abstract}

\maketitle
\tableofcontents

\section*{Introduction}

One of the goals in this paper is to give an explicit formula for the Chern character in higher twisted K-theory.  Such a formula may be used in string theory to compute the charge of a D-brane $x\in K^{0}_{H}(M)$ for a spacetime $M$ in the presence of a flux $H$, according to the equation $$charge(x)= \sqrt{\widehat A(M)} \cup Ch_H(x) \in H^{even}_H(M),$$ cf. \cite{Min-Moo}.  While it is well known to experts that the Chern character abstractly exists, as may be seen using methods of rational homotopy theory cf. \cite{AS2,FHT,Teleman}, it turns out to be rather tricky to give an explicit formula for the Chern character in the general case.  In the end we show that for every odd degree cohomotopy class of a manifold $M$ (which defines an odd degree cohomology class $H$ of $M$) we can define an $H$-twisted K-theory and produce a formula for the Chern character in this case. As explained below, higher twisted K-theory and higher twisted cohomology arise in String/M-theory, and the Chern character constructed relates these. We begin with a brief historical introduction.

$D$-branes, studied in \cite{Pol1, Pol2} are extended objects upon which open strings can end with Dirichlet boundary conditions and are fundamental to string theory. 
It was argued in \cite{Min-Moo, W98} that charges of $D$-branes take values in $K$-theory. In \cite{BM00}, twisted K-theory was proposed as a classification for D-branes in a background flux. The proposal was consolidated in \cite{BCMMS,MS}, where geometric representatives for twisted K-theory were defined, and so was the twisted Chern character. Later in \cite{AS1} an alternate description of twisted K-theory was presented, and an alternate construction of the twisted Chern character presented in \cite{AS2}.

For an account of T-duality in the absence of a background flux, using K-theory, see \cite{Hori}.
In \cite{BEM04a, BEM04b}, T-duality in the presence of a background flux was studied for spacetimes compactified on a principal 
circle bundle $Z$. It was established that there is a change in topology whenever either the principal circle bundle is nontrivial or 
the background flux is topologically nontrivial.  It was also first established that T-duality gives an isomorphism of twisted K-theory groups, 
\begin{align*}
T: K^{*}_{H}(Z) &\to K^{*+1}_{\hat{H}}(\hat{Z}),
\end{align*}
where $\hat{Z}$ is the T-dual principal circle bundle.
In \cite{BS, West, MR} there are alternate approaches to proving this; via a T-duality classifying space and via noncommutative geometry.

In \cite{BEM14} the authors first used higher twisted K-theory and higher twisted cohomology in String Theory. 
The authors argue that the 7-twisted K-theory/cohomology  classify certain conserved charges in 
type IIB supergravity.  They then conclude that spherical T-duality, which is duality of principal $SU(2)$-bundles 
with degree 7 flux, provides a one to one map between conserved charges in certain 
topologically distinct compactifications and also a novel electromagnetic duality on the fluxes.
They also suggest that spherical T-duality preserves the spectra of certain spherical 3-branes that wrap 
$S^3$ cycles in some spacetime $X$, i.e.\ by replacing closed strings, described by $\text{Maps}(S^1,X)$, by 
spherical 3-branes described by $\text{Maps}(S^3,X)$. Spherical T-duality in case of non-principal 
$SU(2)$-bundles is studied in \cite{BEM142}. 
There are several models for higher twisted K-theory, cf. \cite{Guerra, Madsen, Segal, Teleman}.
The model of higher twisted K-theory in \cite{Pennig} using 
the topological K-theory of certain operator algebra bundles will be used here.

Here we consider oriented $S^{2k-1}$-bundles $Z$ over $2k$ dimensional closed oriented manifolds $M$. 
Let $H$ be a top degree twist on $Z$, $H \in H^{4k-1}(Z, \ZZ)\cong \ZZ$. Assuming that $M$ is a torsion-free manifold, then we can compute the 
cohomology and K-theory of $Z$ using the Gysin sequence. We can also compute the $H$-twisted K-theory and $H$-twisted cohomology of 
$Z$. Suppose that the Euler class of $Z$ is $e\in H^{2k}(M, \ZZ) \cong \ZZ$. 
Note that the pushforward map $\pi_*: H^{4k-1}(Z, \ZZ) \to H^{2k}(M, \ZZ) $ is an isomorphism, where $\pi:Z\to M$
is the projection map.
Define a spherical T-dual of $(Z, H)$ to be a pair 
$(\hat Z, \hat H)$ where the Euler class of $\hat Z$ is equal to $\pi_*(H)$. Let $\hat\pi:\hat Z \to M$ be the projection
where again the pushforward map $\hat\pi_*: H^{4k-1}(\hat Z, \ZZ) \to H^{2k}(M, \ZZ) $ is an isomorphism.
Define $\hat\pi_*(\hat H)$ to be equal to the Euler class $e$. Then we show that there are degree shifting isomorphisms $K^\bullet_H(Z) \cong  K^{\bullet+1}_{ \hat H}(\hat Z)$ and $H^\bullet_H(Z, \ZZ) \cong  H^{\bullet+1}_{\hat H}(\hat Z, \ZZ)$.
Furthermore, in the case that $M$ is not torsion free, our Chern character formula allows one to obtain a rational degree shifting isomorphism 
$K^\bullet_H(Z) \otimes \mathbb{Q} \cong  
K^{\bullet+1}_{ \hat H}(\hat Z) \otimes \mathbb{Q}$ of the higher twisted 
K-theories.

In \cite{LSW}, they instead consider iterated higher algebraic K-theory and its higher twists. They also define
an abstract Chern character which takes values in the cohomological counterpart of  twisted iterated higher algebraic K-theory,
and prove an abstract T-duality for certain spherical bundles with flux. It is hard to 
relate those results to the results of our paper except to say that they have a similar flavour. 

Higher twisted K-theory arises because K-theory admits natural automorphisms given by tensoring with virtual line-bundles,
that is virtual vector bundles of rank 1. The main construction is that of the Chern character in higher twisted K-theory for twists arising from cohomotopy classes. 
We generalise and clarify some constructions in \cite{AS2}.
The relevant Atiyah-Hirzebruch spectral sequence in higher twisted K-theory is constructed, and, after  taking a detour through singular de Rham theory to establish some essential facts, we also construct the relevant spectral sequence for higher twisted cohomology. We conclude that for for any higher twist arising from the cohomotopy set of the space, the Chern character can be constructed and is a real isomorphism between higher twisted K-theory and higher twisted cohomology.  The Atiyah-Hirzebruch spectral sequence we give is instrumental to the spherical T-duality computations.\\

\noindent {\bf Acknowledgements}\\
HS and VM were partially supported by funding from the Australian Research Council, through the Australian Laureate Fellowship FL170100020.
LM and VM were partially supported by funding from the Australian Research Council, through the Discovery Project grant DP200100729.
HS wishes to acknowledge Jarah Evslin for useful discussions about spectral sequences and their differentials. \\
All three authors would like to thank the two anonymous referees for their comments on the paper, which have helped to improve the exposition.

\section{Preliminaries}\label{prelims}

To make the paper more self contained, we add some preliminaries on operator algebras 
that are used in the paper. The main references to this section are the books \cite{WO, Black} and the first two sections of the thesis \cite{Brook}.

We begin with the definition of the Cuntz algebra $\cO_\infty$. It is the universal $C^*$-algebra generated by infinitly many generators $\{S_1, S_2, \ldots\}$ subject to 
the relations $S_i^*S_i=1$ for all $i$, and $\sum_{j=1}^N S_j S_j^* \le 1$ for all $N\in \mathbb N$. It is a {\em simple} $C^*$-algebra with many important algebraic properties that have been thoroughly studied
in the literature.

We consider the short exact sequence of $C^*$-algebras,
$$
0\to \cO_\infty\otimes \cK \to \cM(\cO_\infty\otimes \cK ) \stackrel{\pi}{\to}  \cM(\cO_\infty\otimes \cK )/\cO_\infty\otimes \cK \to 0
$$
where $\cK$ denotes the algebra of compact operators on an infinite dimensional Hilbert space, 
$\cM$ denotes the multiplier algebra . An operator $F\in  \cM(\cO_\infty\otimes \cK ) $ is said to be $\cO_\infty\otimes \cK$-Fredholm if 
$\pi(F) \in GL( \cM(\cO_\infty\otimes \cK )/\cO_\infty\otimes \cK)$, the group of invertible elements in $ \cM(\cO_\infty\otimes \cK )/\cO_\infty\otimes \cK$.
Then the space of all $\cO_\infty\otimes \cK$-Fredholm operators
$$
Fred_{\cO_\infty\otimes \cK } = \pi^{-1} (GL( \cM(\cO_\infty\otimes \cK )/\cO_\infty\otimes \cK))
$$
is an open subset of the $C^*$-algebra $\cM(\cO_\infty\otimes \cK )$, so it is a smooth Banach manifold.

From the short exact sequence above, one can deduce that 
$$
K_0 ( \cO_\infty\otimes \cK) \cong \pi_0(Fred_{\cO_\infty\otimes \cK })\cong \ZZ,
$$
since Cuntz \cite{Cuntz81} has shown that $ \cO_\infty\otimes \cK$ is KK-equivalent to $\CC$. In fact, we deduce that 
$$
\pi_j(Fred_{\cO_\infty\otimes \cK })\cong \ZZ\quad\text{if $j$ is even},\quad
\cong 0 \quad \text{if $j$ is odd}$$

Pennig's model \cite{Pennig} of the higher twisted K-theory of a space $X$ uses principal $Aut(\cO_{\infty}\otimes\cK)$-bundles over $X$, which are classified by maps $\lambda:X\rightarrow BAut(\cO_{\infty}\otimes\cK)$.  By the associated bundle construction, any such principal bundle $P_{\lambda}$ is associated to an algebra bundle $\cA_{\lambda}$ over $X$ with fibre $\cO_\infty\otimes \cK$.
Then the higher twisted K-theory of $X$ is by definition
\beq
K^j_{\lambda}(X) = K_j(C_0(X, \cA_{\lambda}))
\eeq
where $C_0(X, \cA_{\lambda})$ denotes the $C^*$-algebra of sections of $\cA_{\lambda}$ that vanish at infinity. The associated bundle construction may also be used to yield bundles $\Omega^{j}Fred_{\lambda}$ over $X$ with typical fibre $\Omega^{j}Fred_{\cO_{\infty}\otimes\cK}$ (here $\Omega^{j}$ denotes the iterated based loop space, while $Aut(\cO_{\infty}\otimes\cK)$ acts on $\Fred_{\cO_{\infty}\otimes\cK}$ by conjugation).  Using the short exact sequence above, one can easily show that higher twisted K-theory can be equivalently described as
\beq
K^j_{\lambda}(X) = \pi_0(C_1(X, \Omega^j Fred_{\lambda}))
\eeq
where $C_1(M,\Omega^j  Fred_{\lambda})$ denotes the sections of the bundle $\Omega^j Fred_{\lambda}$ that are 
equal to the identity outside of a compact set.  If $A$ is a closed subset of $X$, then the higher twisted relative K-theory is defined by
\[
K^{j}_{\lambda}(X,A):=K^{j}_{\lambda}(X\setminus A).
\]
There is a long exact sequence associated to relative higher twisted K-theory, which will be discussed later.

An important property of higher twisted K-theory, is that of Bott periodicity, 
$$
K^j_{\lambda}(X) \cong K^{j+2}_{\lambda}(X).
$$
This says that it is enough to study $K^0_{\lambda}(X)$ and $K^1_{\lambda}(X)$.  We remark here that from here on, \emph{indices in $\ZZ_{2}$ will always be shown in bold font}.  Thus the $\ZZ_{2}$-graded $K$-theory groups will be denoted $K^{\mathbf{0}}_{\lambda}(X)$ and $K^{\mathbf{1}}_{\lambda}(X)$.

One final remark on higher twisted K-theory is that twists on a space $X$ in the general sense outlined above \emph{need not come from cohomology classes} for $X$.  This is to be contrasted to ordinary twisted K-theory, in which twists always arise from integral cohomology classes.  In the special cases where the twist \emph{does} correspond to a cohomology class $[H]\in H^{*}(X;\ZZ)$, the higher twisted K-theory will be denoted
\[
K^{j}_{[H]}(X)
\]
instead of $K^{j}_{\lambda}(X)$.  Examples of this include when $X$ is a finite, connected, torsion-free CW complex \cite[Corollary 4.4]{DP} (and in particular whenever $X$ is a compact, torsion-free, connected manifold).  
We recall from the work of M. Dadarlat and U. Pennig in \cite{DP} that the homotopy groups for 
$BAut(\mathcal{O}_{\infty}\otimes \mathcal{K})$ were computed to be
$\ZZ_2$ in degree $1$, 
$\ZZ$ in all other odd degrees and $0$ otherwise.
An important class of examples for us are those for which the twist $\lambda:X\rightarrow BAut(\cO_{\infty}\otimes\cK)$ factors through an odd sphere, so that $\lambda = f\circ H$, where $f:S^{2k+1}\rightarrow BAut(\cO_{\infty}\otimes\cK)$ is the generator of
the $(2k+1)$ homotopy group, and where $H:X\rightarrow S^{2k+1}$ defines a class in the cohomotopy set of $X$, hence a class $[H]\in H^{2k+1}(X;\ZZ)$.  Note that by the arguments of \cite{AS1}, such cohomotopy twists generalise the ordinary twists already familiar from string theory.

Higher twisted cohomology was studied in detail in \cite{MW}, and will also be discussed later in the paper.

\section{Higher Twisted K-theory and the Chern Character}\label{htk_ch}

The goal of this section is to define a Chern character for higher twisted K-theory, when the twist is given by a cohomotopy class. 
We start by using work of M. Dadarlat and U. Pennig in \cite{DP} to construct a $Fred_{\mathcal{O}_{\infty}\otimes \mathcal{K}}$-bundle over odd dimensional spheres $S^{2k+1}$, which we denote by $Y_{2k+1}$ for $k \geq 1$. Generalising the method of Atiyah and Segal in \cite{AS1}, we then show how to construct a Chern character for the bundle $Y_{2k+1}$. This Chern character takes values in a higher twisted cohomology group, and has the property that it lifts the classical Chern character of $Fred_{\mathcal{O}_{\infty}\otimes \mathcal{K}}$. 
Finally, using the constructed Chern character for the bundle $Y_{2k+1}$, we are able to construct a Chern character for a compact smooth manifold $X$, with a fixed cohomotopy class 
$[\lambda] \in [X, S^{2k+1}] $ as twist. Analogous to the case of $Y_{2k+1}$, the Chern character for $X$ lands in a twisted Cohomology group, determined by the cohomotopy twist 
$[\lambda]$.
Our strategy can be seen as a generalisation of the methods employed by Atiyah and Segal in \cite{AS1}, where a Chern character for 3-twisted K-theory is constructed by first constructing one over a universal $Fred$-bundle over $S^3$.

\subsection{The universal $Fred_{\cO_\infty\otimes \cK }$-bundle over $S^{2k+1}$}\label{univ_fib}

In this subsection, we will be dealing with 
$Fred_{\mathcal{O}_{\infty}\otimes \mathcal{K}}$-bundles over compact smooth manifolds. Such bundles are classified by maps into the classifying space
$BAut(Fred_{\mathcal{O}_{\infty}\otimes \mathcal{K}})$, and can be given as the pull-back of the universal 
$Fred_{\mathcal{O}_{\infty}\otimes \mathcal{K}}$-bundle over 
$BAut(Fred_{\mathcal{O}_{\infty}\otimes \mathcal{K}})$, which we will denote by 
$Fred(P)$. The reader is invited to consult section \ref{prelims} for preliminaries about such bundles, and the references there in for a more comprehensive discussion.

The cohomology generator $u_{2k+1} \in H^{2k+1}(S^{2k+1};\ZZ)$ gives rise to a
$Fred_{\mathcal{O}_{\infty}\otimes \mathcal{K}}$-bundle, denoted $Y_{2k+1}$, via corollary 4.7 (ii) in \cite{DP}, and the
associated bundle construction, see section \ref{prelims}. In fact, from the referenced corollary, we see that every cohomology class in 
$H^{2k+1}(S^{2k+1};\ZZ)$ gives rise to a 
$Fred_{\mathcal{O}_{\infty}\otimes \mathcal{K}}$-bundle, and the one corresponding to the generator $u_{2k+1}$ can be thought of as generating all these bundles. In this regard, we can think of $Y_{2k+1}$ as a universal 
$Fred_{\mathcal{O}_{\infty}\otimes \mathcal{K}}$-bundle over $S^{2k+1}$.

It will be important for us to understand the Serre
spectral sequence associated to the bundle 
$Y_{2k+1} \rightarrow S^{2k+1}$. 
As the base of the Serre spectral sequence is $S^{2k+1}$, we see that the 
only non-trivial differential is 
$d_{2k+1} : E_{2k+1}^{0, q} \rightarrow E_{2k+1}^{2k+1, q-2k}$, and the 
spectral sequence reduces to the Wang exact sequence.

\begin{tikzcd} 
\cdots
\arrow[r] & H^{q-1}(F) \arrow[r, "d"] & H^{q-2k}(F)  \arrow[r] &
 H^q(Y_{2k+1}) \arrow[r] & H^q(F) \arrow[r] & \cdots \\ 
\end{tikzcd}

where $d = d_{2k+1}$ is the differential of the Serre spectral sequence
and $F= Fred_{\mathcal{O}_{\infty}\otimes \mathcal{K}}$.

Any $Fred_{\mathcal{O}_{\infty}\otimes \mathcal{K}}$-bundle, over 
$S^{2k+1}$, can be constructed by taking two trivial bundles over each 
$D^{2k+1}$ hemisphere of $S^{2k+1}$, and gluing them together via a gluing map.
In other words, such a bundle is determined via its clutching map.
\begin{equation*}
\psi : S^{2k} \times Fred_{\mathcal{O}_{\infty}\otimes \mathcal{K}}
\rightarrow S^{2k} \times Fred_{\mathcal{O}_{\infty}\otimes \mathcal{K}}.
\end{equation*}

The following proposition gives the action of $\psi^*$ on cohomology.

\begin{proposition}\label{clutch_action_cohom}
Let $F= Fred_{\mathcal{O}_{\infty}\otimes \mathcal{K}}$. For an $x \in H^n(F)$ we have that 
$\psi^*(1 \otimes x) = 1\otimes x + w\otimes dx$ 
\end{proposition}
\begin{proof}
We have the following commutative diagram

\begin{tikzcd} 
H^0(S^{2k}) \otimes H^n(F) \oplus H^{2k}(S^{2k}) \otimes H^{n-2k}{F}
\arrow[r, "\psi^*"] & 
H^0(S^{2k}) \otimes H^n(F) \oplus H^{2k}(S^{2k}) \otimes H^{n-2k}{F} 
\arrow[d, "\pi"] \\
 H^{n}(F) \arrow[u, "i"]  \arrow[r,"d"] &
 H^{n-2k}(F)  \\ 
\end{tikzcd}

where we have used Kunneth's theorem to write $H^n(S^{2k} \times F) = 
H^0(S^{2k}) \otimes H^n(F) \oplus H^{2k}(S^{2k}) \otimes H^{n-2k}{F}$, 
$i(x) = 1\otimes x$ denotes the inclusion and $\pi$ the projection. 
 
On taking the trivial clutching function, we see that the first term of
$\psi^*(1 \otimes x)$ must be $1 \otimes x$. Together with the above commutative 
diagram, it follows that 
$\psi^*(1 \otimes x) = 1\otimes x + w\otimes dx$.
\end{proof}

Given a compact space $X$ it is well known that the Picard group
$Pic(X)$, formed by (equivalence classes of) line bundles on $X$, is a multiplicative
subgroup of $K(X)$. Hence, there exists an action
\begin{equation*}
Pic(X) \times K(X) \rightarrow K(X)
\end{equation*}
by group homomorphisms. 
On the level of vector bundles, this action is induced by tensoring with a line bundle: $(L, E) \mapsto L \otimes E$. 

Atiyah and Segal, see p. 5-7 in \cite{AS1}, were able to use this action to understand the differential $d_3$ in the Serre spectral sequence of a universal $Fred$-bundle over $S^3$. This provided a convenient strategy for them to define the Chern character for twisted $K$-theory twisted by a degree $3$ cohomology class. We will be taking a similar approach, and as the bundles we are working with are slightly more general. We will need to understand a more general action.

The above action is a special case of a more general action given by 
virtual line bundles, which we now describe.
Consider the rank function
\begin{equation*}
rk : K(X) \rightarrow \ZZ
\end{equation*}
that gives the rank of a virtual vector bundle in $K(X)$. 

Let $VPic(X) := rk^{-1}(1)$. An element of $VPic(X)$ is a virtual line bundle. We point out that 
while $Pic(X)$ is a multiplicative group, $VPic(X)$ is only a set. An element of
$VPic(X)$ does not necessarily have a multiplicative inverse.
We call $VPic(X)$ the \textit{virtual Picard set} on $X$. These facts give us an action
\begin{equation*}
VPic(X) \times K(X) \rightarrow K(X)
\end{equation*}
of virtual line bundles on the $K$-theory of $X$, induced via tensor product.

The definition of the virtual Picard set looks very similar to that of the 
Picard group. However, there is a key difference. It is well know that there 
exists an isomorphism
\begin{equation*}
c_1 : Pic(X) \rightarrow H^2(X;\ZZ)
\end{equation*}
induced via the first Chern class. This implies that $Pic(X)$ is a cohomological 
action on $K(X)$. In the case of $VPic(X)$, there is no such isomorphism. In 
other words, the action of $VPic(X)$ on $K(X)$ is not cohomological in any way, 
and is an honest action coming from the $K$-theory of $X$ as opposed to the 
singular cohomology of $X$. This action will be used shortly to understand the differential $d_{2k+1}$ in the Serre spectral sequence of the bundle $Y_{2k+1}$.

%

Since for compact spaces $X$, $K^0(X) \cong [X, Fred_{\mathcal{O}_{\infty}\otimes \mathcal{K}}]
 \cong [ X, Fred]$, we can use Yoneda's lemma to conclude that $Fred_{\mathcal{O}_{\infty}\otimes \mathcal{K}}$
 and $Fred$ are homotopy equivalent.
 Therefore cohomology ring of $Fred_{\mathcal{O}_{\infty}\otimes \mathcal{K}}$ is given by the following proposition.
\begin{proposition}\label{ring_structure_Fred}
We have an isomorphism of rings 
$H^*(Fred_{\mathcal{O}_{\infty}\otimes \mathcal{K}}; \ZZ) \cong 
\ZZ[s_1, s_2,\ldots, s_n,\ldots]$, where 
$s_i \in H^{2i}(Fred_{\mathcal{O}_{\infty}\otimes \mathcal{K}};\ZZ)$.
\end{proposition}

The element $s_0 \in H^0(Fred_{\mathcal{O}_{\infty}\otimes \mathcal{K}}; \ZZ)$ corresponds to the index. Letting 
$Fred^k_{\mathcal{O}_{\infty}\otimes \mathcal{K}}$ denote the component of 
$Fred_{\mathcal{O}_{\infty}\otimes \mathcal{K}}$ consisting of index $k$ operators, we have that $s_0 = k \in 
H^0(Fred^k_{\mathcal{O}_{\infty}\otimes \mathcal{K}}; \ZZ)$. In particular, in the index zero component $s_0 = 0$.

We will also need the following theorem of Bott, see corollary 9.8 in 
\cite{Husemoller}.
\begin{theorem}[Bott]\label{Bott's_thm}
Let $u \in H^{2n}(S^{2n};\ZZ)$ be a generator. Given any rank $n$ complex vector bundle $E$ over $S^{2n}$ we have  
$c_n(E) = m(n-1)!u$ for some $m \in \ZZ$, where $c_n(E)$ denotes the nth Chern class of $E$.
\end{theorem}

We remark that $\widetilde{K}(S^{2n})$ can be generated by an element of 
the form $E - \underline{n}$, where $E$ is a rank $n$ complex vector bundle and
$\underline{n}$ denotes a trivial rank $n$ bundle so
 the Chern character of $E - \underline{n}$ must be non-zero. So in the above theorem, we can always find a $E$ with 
$c_n(E) \neq 0$.

Bott's theorem will be applied to compute the differential $d_{2k+1}$ in 
the Serre spectral sequence associated to the bundle 
$Y_{2k+1} \rightarrow S^{2k+1}$. Before we can show how this is done, we give the reader a reminder of the theory of Chern classes and the Chern character.

Given a rank $n$ vector bundle, we define the Chern classes via the formal variables $x_1,\ldots,x_n$ so that
\begin{equation}\label{chern_classes}
\sum_{i=0}^nc_i = \prod_{i=1}^n(1+x_i)
\end{equation}
and define the power sums $s_n$ by
\begin{equation}\label{power_sums}
s_n = \sum_{i=1}^nx_i^n.
\end{equation}

The importance of the power sums comes from the fact that they give a nice definition for the Chern character of a vector bundle $E$
\begin{equation*}
Ch(E) = \sum_{n=0}^{\infty}\frac{s_n}{n!}.
\end{equation*}

Given two vector bundles $E$ and $F$ we have that 
$Ch(E \otimes F) = Ch(E)Ch(F)$, and hence we obtain the formula
\begin{equation*}
s_n(E\otimes F) = {n\choose 0}s_n(E)s_0(F) + \cdots + 
{n\choose k}s_{n-k}(E)s_k(F) + \cdots + s_0(E)s_n(F). 
\end{equation*}

A particular special case of the above formula will be important in our context. Fix a rank $k$ complex vector bundle $E$ over a $2k$-dimensional manifold $X$, such that $c_m(E) = 0$ for all $1 \leq m \leq k-1$ and 
$c_k(E) = v \in H^{2k}(X;\ZZ)$. Then for any $F \in K(X)$, letting 
$s_n$ denote the power sums for $F$ and $s_n(u)$ the power sums for 
$E \otimes F$, we have
\begin{equation}\label{power_sum_tensor}
s_n(u) = s_n + \frac{(-1)^{k+1}n!}{(k-1)!(n-k)!}vs_{n-k}.
\end{equation}

We can use this to obtain a formula for the differential $d_{2k+1}$ in the Serre spectral sequence for the bundle $Y_{2k+1}$ over $S^{2k+1}$.

\begin{proposition}\label{sss_differential_formula}
In the Serre spectral sequence for the bundle  
$Y_{2k+1} \rightarrow S^{2k+1}$, with fibre $Fred^j_{\mathcal{O}_{\infty}\otimes \mathcal{K}}$ and cohomology with $\RR$-coefficients, the differential 
$d_{2k+1}$ is given on the cohomology generators $s_n$ by
\begin{align*}
d_{2k+1}s_n &= \frac{(-1)^{k+1}n!}{(n-k)!}mu_{2k+1}s_{n-k} \text{ for } n \neq k \\
d_{2k+1}s_k &= (-1)^{k+1}k!mu_{2k+1}s_{0} = (-1)^{k+1}k!mju_{2k+1}
\end{align*}
where $u_{2k+1} \in H^{2k+1}(S^{2k+1};\ZZ)$ is a generator and $0 \neq m \in \ZZ$.
\end{proposition}
\begin{proof}
Let $u_{2k} \in H^{2k}(S^{2k};\ZZ)$ be a generator that suspends to the generator 
$u_{2k+1} \in H^{2k+1}(S^{2k+1};\ZZ)$.  Applying Bott's theorem \ref{Bott's_thm}, we can find a rank $n$ bundle $E$ such that 
$c_n(E) = m(n-1)!u$, for some integer $m \neq 0$. We then consider the virtual line bundle
$\widetilde{E} := E - \underline{n-1}$, where $\underline{n-1}$ denotes the trivial rank $n-1$ bundle on $S^{2k}$.

By tensoring with $\RR$, we let $\widetilde{E}_{\RR}$ denote the $1$-dimensional 
$\RR$-subspace of $VPic(S^{2k})\otimes \RR$ generated by $\widetilde{E} $. Then $c_n$ extends to an isomorphism
\begin{equation}\label{c_n_iso}
c_n : \widetilde{E}_{\RR} \rightarrow H^{2k}(S^{2k}; \RR).
\end{equation}

The above map,
together with the fact 
$Fred_{\mathcal{O}_{\infty}\otimes \mathcal{K}}$ represent a classifying space 
for $K$-theory, imply that cohomologically the action of the clutching map is 
given by formula \ref{power_sum_tensor} with $v = m(k-1)!u$.

Applying proposition \ref{clutch_action_cohom}, we obtain the formula
for $d_{2k+1}$. 

\end{proof}

Observe that the spectral sequence gives rise to the exact sequence

\begin{tikzcd} 
H^{2k}(Fred^j_{\mathcal{O}_{\infty}\otimes \mathcal{K}}) 
\arrow[r, "d_{2k+1}"]
 & 
H^{2k+1}(X)
\arrow[r, "\pi^*"] &
 H^{2k+1}(Fred^j_{\mathcal{O}_{\infty}\otimes \mathcal{K}})
\end{tikzcd}

where $\pi : Y_{2k+1} \rightarrow S^{2k+1}$ denotes the projection.
We thus see that 
\begin{align*}
\pi^*d_{2k+1}s_k &= \pi^*\bigg{(}(-1)^{k+1}k!mju_{2k+1}\bigg{)} \\
&= (-1)^{k+1}k!mj\pi^*u_{2k+1} \\
&= 0.
\end{align*}
This implies $\pi^*(u_{2k+1}) = 0$ if $j \neq 0$. If the bundle 
$Y_{2k+1}$ posses a section then
$\pi^*$ is injective. As $u_{2k+1} \neq 0 \in H^{2k+1}(S^{2k+1};\ZZ)$ it follows that 
$j = 0$. Therefore we find that there are no sections with non-zero
index. This means that when we define higher twisted K-theory for a smooth compact manifold $X$ twisted by a cohomotopy class, represented by a map $\lambda : X \rightarrow S^{2k+1}$,
we need only focus on the 
$Fred^0_{\mathcal{O}_{\infty}\otimes \mathcal{K}}$-component of the 
bundle $Y_{2k+1}$.
This is similar to what Atiyah and Segal observe for 3-twisted $K$-theory, see p. 7 
in \cite{AS1}.

\subsection{The Chern character for the universal $Fred_{\cO_\infty\otimes \cK}$-bundle over $S^{2k+1}$ }\label{chern_character_odd_sphere}

In this subsection, we will define a Chern character for the bundle $Y_{2k+1}$ that will land in a higher twisted cohomology group. This Chern character can be seen as a lift of the classical Chern character on 
$Fred_{\mathcal{O}_{\infty}\otimes \mathcal{K}}$, and will play a pivotal role in the next subsection where we define a Chern character for smooth compact manifolds with a fixed cohomotopy twist.

We recall that $s_n \in H^{2n}(Fred_{\mathcal{O}_{\infty}\otimes \mathcal{K}})$ 
represents a cohomological generator, given by proposition \ref{ring_structure_Fred}.

We will let $\eta = \pi_{2k+1}^*u_{2k+1} \in H^{2k+1}(Y_{2k+1};\RR)$, where 
$u_{2k+1} \in H^{2k+1}(S^{2k+1};\RR)$ is a generator represented by a closed form whose support does not contain the poles of $S^{2k+1}$. 
Furthermore, 
let $F_{\infty}$ denote the fibre of $Y_{2k+1}$ over the north pole of 
$S^{2k+1}$, which we denote by $\{\infty\}$, 
and $F_0$ the fibre of $Y_{2k+1}$ over the south pole, which we denote by 
$\{0\}$.

As mentioned at the end of section \ref{univ_fib}, there are no sections of non-zero index in higher twisted K-theory. This means it suffices to restrict to the 
$Fred^0_{\mathcal{O}_{\infty}\otimes \mathcal{K}}$-component of the 
bundle $Y_{2k+1}$, and we shall do this from here on in.

In order to construct a Chern character, we will start by constructing  forms $S_n$ on $Y_{2k+1}$, for $n \geq 0$, that satisfy the following two conditions:
\begin{itemize}
\item[(1)] $S_n\vert_{F_0} = s_n$
\item[(2)] $dS_n = \frac{(-1)^{k+1}n!}{(n-k)!}m\eta S_{n-k}
= \frac{(-1)^{k+1}n!}{(n-k)!}\eta_{2k+1} S_{n-k}$
\end{itemize}
where we remind the reader that the coefficient of the right hand side of (2) comes from proposition \ref{sss_differential_formula}, and we have defined
$\eta_{2k+1} = m\eta$.

The construction proceeds via induction on $n$. We define $S_0 = 0$. Suppose we have constructed $S_m$ for all $m < n$ that satisfies the above two conditions. Our goal is to construct $S_n$. 

The bundle
$Y_{2k+1}$ can be trivialised over $S^{2k+1}\backslash\{\infty\}$, we have that 
$Y_{2k+1}\backslash F_{\infty} \cong D^{2k+1}\times 
Fred_{\mathcal{O}_{\infty}\otimes \mathcal{K}}$. In particular, we can pull back
the form $s_n$ on $Fred_{\mathcal{O}_{\infty}\otimes \mathcal{K}}$, obtaining a form $s_n'$ on $Y_{2k+1}\backslash F_{\infty}$. 

Let $\rho$ be a bump function on $Y_{2k+1}$ supported away from $F_{\infty}$.
Define $S_n' := \rho s_n'$. By definition, $S_n'$ is an $2n$ form defined on all 
of $Y_{2k+1}$. However, $S_n'$ is not necessarily closed on all of $Y_{2k+1}$. 
We do 
have that $dS_n' = 0$ outside a small neighbourhood of $F_{\infty}$ (i.e. in a small neighbourhood of $F_0$). As $d(dS_n') = 0$ globally on $Y_{2k+1}$, and 
$dS_n' = 0$ in a neighbourhood of $F_0$, we have that
$dS_{n}'$ defines a class in the relative group 
$H^{2n+1}(Y_{2k+1}, F_{0})$. By proposition \ref{sss_differential_formula} this class can also be represented by $\frac{(-1)^{k+1}n!}{(n-k)!}m\eta S_{n-k}$.
It follows that on the level of forms we have
\begin{equation*}
dS_n' - \frac{(-1)^{k+1}n!}{(n-k)!}m\eta S_{n-k} = d\omega_n
\end{equation*} 
where $\omega_n$ is a $2n$-form on $Y_{2k+1}$ that is supported away from 
$F_{\infty}$. We then define $S_n := S_n' - \omega_n$. It is clear that 
$S_n$ satisfies conditions (1) and (2).

We can then define an even form $Ch^{\textbf{0}}_{2k+1}$ on $Y_{2k+1}$ by
\begin{equation}\label{chern_character_on_sphere}
Ch^{\textbf{0}}_{2k+1} = \sum_{n=1}^{\infty}\frac{S_n}{n!}.
\end{equation}
We observe that $D_{\eta_{2k+1}}Ch^{\textbf{0}}_{2k+1} = (d-\eta_{2k+1})Ch^{\textbf{0}}_{2k+1} = 0$, and hence 
$Ch^{\textbf{0}}_{2k+1}$ defines 
a cocycle for higher twisted cohomology, with twist given by the cohomology class $[\eta_{2k+1}]$, of the bundle
$Y_{2k+1}$. Therefore
$Ch^{\textbf{0}}_{2k+1} \in H^{\textbf{0}}_{\eta_{2k+1}}(Y_{2k+1})$, the even twisted cohomology group, and from 
condition (1) above we see that it lifts the classical Chern character in 
$H^*(Fred_{\mathcal{O}_{\infty}\otimes \mathcal{K}})$.


\subsection{The Chern character for smooth manifolds with a cohomotopy twist}\label{chern_character_general}

In this subsection we use the Chern character from the previous subsection to construct a Chern character for smooth compact manifolds $X$. Our Chern character has the restriction that it can only be defined for twists corresponding to cohomotopy classes $[X, S^{2k+1}]$. This latter set is known as the $2k+1$ cohomotopy set of $X$. In general, it is just a set and not a group. However,
we point out that if $X$ is a finite CW-complex of dimension $n$, where $n \leq 2k+1$, then the cohomology
group $H^{2k+1}(X;\ZZ)$ is given by the cohomotopy classes 
$[X, S^{2k+1}]$. This follows from the fact that elements 
in $H^{2k+1}(X;\ZZ)$ are represented by maps into the Eilenberg-Maclane space
$K(\ZZ,2k+1)$, and by cellular approximation such a map can be assumed to be contained in the $2k+1$-skeleton of $K(\ZZ,2k+1)$, which is $S^{2k+1}$. Furthermore, if $X$ is a finite CW complex of dimension $n \leq 2k-2$, then by work of Borsuk, see \cite{borsuk_1} and chapter 2 section 11 in 
\cite{borsuk_2}, it is known that $\pi^k(X)$ is an abelian group. In our case, $X$ will always be a smooth compact manifold, hence can be given the structure of a finite CW complex. Thus if $dim(X) \leq 2k+1$, we have that
$H^{2k+1}(X;\ZZ) = [X, S^{2k+1}]$ and hence any twist corresponding to a class in $[X, S^{2k+1}]$ is necessary cohomological. However, if $dim(X) > 2k+1$ then elements of $[X, S^{2k+1}]$ could determine twists that are not necessarily cohomological. One final point to make is that our Chern character we define for $X$ will be defined via a pull back of the smooth form $Ch^{\textbf{0}}_{2k+1}$. Therefore, we will need to restrict ourselves to smooth maps from
$X$ to $S^{2k +1}$. This will not cause any problems as each class in $[X, S^{2k+1}]$ has a (not necessarily unique) smooth representative. When taking representatives of such cohomotopy classes, we will always be taking smooth ones without saying.

We fix a cohomotopy class $[\lambda] \in [X, S^{2k+1}]$, represented by a smooth map 
$\lambda : X \rightarrow S^{2k+1}$.  We define $Fred_{\lambda}(X) = \lambda^*(Y_{2k+1})$, which is a 
$Fred_{\mathcal{O}_{\infty}\otimes \mathcal{K}}$-bundle over $X$. The bundle $Fred_{\lambda}(X)$ depends on the representative map $\lambda : X \rightarrow S^{2k+1}$, however its isomorphism class only depends on the homotopy class of $\lambda$.

Following Atiyah and Segal's approach in \cite{AS1}, we define the higher twisted 
K-theory, twisted by the class $\lambda$, as follows

\begin{definition}
Given a smooth compact manifold $X$, and a cohomotopy class
$[\lambda] \in [X, S^{2k+1}]$, represented by $\lambda : X \rightarrow S^{2k+1}$, we define the higher twisted K-theory of $X$,
twisted 
by the cohomotopy class $[\lambda]$ and denoted
$K^{\textbf{0}}_{\lambda}(X)$, to consist of homotopy
classes in $[X, Y_{2k+1}]$ that lift the map $\lambda$. In other words, an element of 
$K^{\textbf{0}}_{\lambda}(X)$ is a homotopy class of a smooth map $F : X \rightarrow Y_{2k+1}$ such
that the following digram commutes

\begin{center}
\begin{tikzcd} 
{}
 & 
Y_{2k+1}
\arrow[d, "\pi"] \\
 X  \arrow[r, swap, "\lambda"] \arrow[ur, "F"] &
 S^{2k+1}. 
\end{tikzcd}
\end{center}
\end{definition}

We remark that an equivalent way to define $K^{\textbf{0}}_{\lambda}(X)$ would be as homotopy classes of sections of the bundle $Fred_{\lambda}(X)$. This follows from the
fact that $Fred_{\lambda}(X) = \lambda^*Y_{2k+1}$. We also remark that as $Fred_{\mathcal{O}_{\infty}\otimes \mathcal{K}}$ has the structure of a smooth Banach manifold, see section \ref{prelims}, it follows that $Y_{2k+1}$ can be given the structure of a smooth Banach manifold. Therefore, the concept of a smooth map 
$F : X \rightarrow Y_{2k+1}$ is well defined.

The group structure on $K^{\textbf{0}}_{\lambda}(X)$ comes from the group
structure of the fibres of the bundle $Y_{2k+1}$. I.e. given $[F]$ and $[G] \in 
K^{\textbf{0}}_{\lambda}(X)$, we define $[F] + [G]$ to be the class $[F + G]$. Note that 
$F + G$ is defined as the map such that $F + G(x) := F(x) + G(x)$, which is well 
defined because $F(x), G(x) \in \pi^{-1}(f(x))$ and hence 
$F(x) + G(x) \in \pi^{-1}(f(x))$. By Bott periodicity, we see that
the groups $K^{\textbf{0}}_{\lambda}(X)$ are the even degree groups of the twisted $K$-theory of $X$. The odd case will be defined in the next subsection.

We now define a twisted Chern character for $X$. We recall towards the end of the previous section we defined a Chern character 
$Ch^{\textbf{0}}_{2k+1} \in H_{\eta_{2k+1}}(Y_{2k+1})$. We also recall that the 
cohomology class $\eta_{2k+1} \in H^{2k+1}(Y_{2k+1})$ was chosen via pulling back a particular generator of $H^{2k+1}(S^{2k+1})$, which we denoted $u_{2k+1}$. Therefore, given a 
class $[F] \in K^{\textbf{0}}_{\lambda}(X)$, represented by a smooth map $F : X \rightarrow Y_{2k+1}$, we have
$F^*(\eta_{2k+1})= F^*\pi^*(u_{2k+1}) = \lambda^*(u_{2k+1})$.

We define the twisted Chern character
\begin{equation*}
Ch^{\textbf{0}}_{\lambda} : K^{\textbf{0}}_{\lambda}(X) \rightarrow H^{\textbf{0}}_{\lambda^*(u_{2k+1})}(X)
\end{equation*}
by $Ch^{\textbf{0}}_{\lambda}([F]) = [F^*Ch^{\textbf{0}}_{2k+1}]$, where $F$ is a smooth representative. The fact that this defines an element of the twisted cohomology group $H_{\lambda^*(u_{2k+1})}(X)$ follows by naturality and the fact that $Ch^{\textbf{0}}_{2k+1}  \in  H^{\textbf{0}}_{\eta_{2k+1}}(Y_{2k+1})$.

\subsection{The odd case}\label{odd_ch}

In the previous subsection we witnessed the construction of a Chern character for cohomotopy twists, which gave us a mapping from the higher twisted K-theory of a smooth compact manifold to its twisted cohomology. Strictly speaking, we constructed an even Chern character that mapped even twisted K-theory to even twisted cohomology. The purpose of this subsection is to extend this construction to the odd case. Namely, we construct an odd Chern character that will map odd twisted K-theory to odd twisted cohomology.

We have already seen that $Aut(\mathcal{O}_{\infty}\otimes \mathcal{K})$ acts on 
$Fred_{\mathcal{O}_{\infty}\otimes \mathcal{K}}$ via conjugation, see section \ref{prelims}. Letting $\Omega Fred_{\mathcal{O}_{\infty}\otimes \mathcal{K}}$ denote the smooth pointed loop space. In other words, an element
$\gamma \in 
\Omega Fred_{\mathcal{O}_{\infty}\otimes \mathcal{K}}$ consists of a smooth pointed map $\gamma : S^1 \rightarrow 
Fred_{\mathcal{O}_{\infty}\otimes \mathcal{K}}$, where the base point of
$Fred_{\mathcal{O}_{\infty}\otimes \mathcal{K}}$ is the identity operator.
We see that this
action extends to an action on 
$\Omega Fred_{\mathcal{O}_{\infty}\otimes \mathcal{K}}$ in the pointwise fashion. In other words, given a continuous loop 
$\gamma : S^1 \rightarrow  Fred_{\mathcal{O}_{\infty}\otimes \mathcal{K}}$ and an element $A \in Aut(\mathcal{O}_{\infty}\otimes \mathcal{K})$. We define
$A\cdot\gamma : S^1 \rightarrow Fred_{\mathcal{O}_{\infty}\otimes \mathcal{K}}$ to be the loop $(A\cdot\gamma)(t) := A(\gamma(t))$, where $A(\gamma(t))$ is the conjugation action of $A$ on the point $\gamma(t) \in Fred_{\mathcal{O}_{\infty}\otimes \mathcal{K}}$.
We point out that
given an $Aut(\mathcal{O}_{\infty}\otimes \mathcal{K})$-bundle we can, via the above action and the associated bundle construction, produce an 
$\Omega Fred_{\mathcal{O}_{\infty}\otimes \mathcal{K}}$-bundle. 

As we saw in section \ref{univ_fib}, the generator 
$u_{2k+1} \in H^{2k+1}(S^{2k+1};\ZZ)$ gives rise to a principal 
$Aut(\mathcal{O}_{\infty}\otimes \mathcal{K})$-bundle via corollary 4.7 (ii) in \cite{DP}. Using this bundle, and the above action on $\Omega Fred_{\mathcal{O}_{\infty}\otimes \mathcal{K}}$, the associated bundle construction produces a  
$\Omega Fred_{\mathcal{O}_{\infty}\otimes \mathcal{K}}$-bundle over $S^{2k+1}$ corresponding to $u_{2k+1}$, which we shall denote by $L_{2k+1}$.

There is an evaluation map 
$e: \Omega Fred_{\mathcal{O}_{\infty}\otimes \mathcal{K}} \times S^1 
\rightarrow Fred_{\mathcal{O}_{\infty}\otimes \mathcal{K}}$ defined by 
$e(\gamma, t) = \gamma(t)$.

Given a cohomology generator $s_n \in 
H^{2n}(Fred_{\mathcal{O}_{\infty}\otimes \mathcal{K}})$ we let 
\begin{equation*}
\Omega s_n := \int_{S^1}e^*s_n \in 
H^{2n-1}(\Omega Fred_{\mathcal{O}_{\infty}\otimes \mathcal{K}}).
\end{equation*}

The above evaluation map induces an evaluation map, which we also denote by $e$
\begin{equation}
e : L_{2k+1} \times S^1 \rightarrow Y_{2k+1}
\end{equation}
as the reader can check. 

We recall from section \ref{chern_character_odd_sphere} that we constructed forms $S_n \in \Omega^{2n}(Y_{2k+1})$. We define 
$T_n := e^*S_n \in \Omega^{2n}(L_{2k+1}\times S^1)$. Integrating over the fibre, we define forms
\begin{equation}
\Omega S_n := \int_{S^1}T_n \in \Omega^{2n-1}(L_{2k+1}).
\end{equation}

We also recall that the forms $S_n$ satisfied the formula 
$dS_n = \frac{(-1)^{k+1}n!}{(n-k)!}\eta_{2k+1} S_{n-k}$, where the form 
$\eta_{2k+1}= m\eta$ is a $2k+1$-form on $Y_{2k+1}$ coming from the pull-back of a certain generating $2k+1$-form on $S^{2k+1}$, see section 
\ref{chern_character_odd_sphere}
for details.

The $2k+1$-form $\eta_{2k+1}$ is constant along the fibres of $Y_{2k+1}$, as it comes from the pull-back of a form on the base via the projection map. In particular, since the evaluation map is fibre preserving, it follows that the form
$e^*(\eta_{2k+1})$ does not have any $d\theta$-components, $\theta$ being a coordinate on $S^1$.

\begin{proposition}
$d\Omega S_n = \lambda(n,k)e^*(\eta_{2k+1})\Omega S_{n-k}$, where
$\lambda(n, k) := \frac{(-1)^{k+1}n!}{(n-k)!}$.
\end{proposition}
\begin{proof}
We differentiate under the integral, which is possible as $S^1$ is compact and all our forms are smooth.
\begin{align*}
d\Omega S_n &= d\int_{S^1}e^*S_n \\
&= \int_{S^1}de^*S_n \\
&= \int_{S^1}e^*dS_n \\
&= \int_{S^1}e^* \bigg{(}\frac{(-1)^{k+1}n!}{(n-k)!}\eta_{2k+1} S_{n-k}\bigg{)} \\
&= \int_{S^1}\frac{(-1)^{k+1}n!}{(n-k)!}e^*(\eta_{2k+1}) e^*(S_{n-k}) \\
&= \frac{(-1)^{k+1}n!}{(n-k)!}e^*(\eta_{2k+1})\Omega S_{n-k}
\end{align*}
where the last equality comes from the fact that $e^*(\eta_{2k+1})$ will not  have any $d\theta$-components, and so does not contribute to the integral.

\end{proof}

We now define a Chern character for $L_{2k+1}$, that will represent the universal odd Chern character for cohomotopy twists of degree $2k+1$, by
\begin{equation}
Ch^{\textbf{1}}_{2k+1} = \sum \frac{\lambda(n,k)\Omega S_n}{n!}.
\end{equation}

It is easy to see that $D_{e^*(\eta_{2k+1})}Ch^{\textbf{1}}_{2k+1} = 
(d - e^*(\eta_{2k+1}))Ch^{\textbf{1}}_{2k+1} = 0$. Hence $Ch^{\textbf{1}}_{2k+1}$ is an element of the odd twisted cohomology 
$H^*_{e^*(\eta_{2k+1})}(L_{2k+1})$.

Given the odd Chern character for $L_{2k+1}$, we can define a general odd Chern character for any space $X$ with a degree $2k+1$-cohomotopy twist. This will be a map from the odd higher twisted K-theory of $X$ to twisted cohomology.

Given a smooth compact manifold $X$ with a cohomotopy class represented by a smooth map 
$\lambda : X \rightarrow S^{2k+1}$, we observe that 
$\Omega Fred_{\lambda}(X) := \lambda^*L_{2k+1}$ is a $\Omega Fred_{\mathcal{O}_{\infty}\otimes \mathcal{K}}$-bundle over $X$. Furthermore, the isomorphism class of this bundle over $X$ depends on the homotopy class of $\lambda$.
This allows us to define the odd twisted K-theory of $X$, corresponding to the twist $\lambda$, in a similar way to the even case done in section \ref{chern_character_general}.

\begin{definition}
Given a smooth compact manifold $X$, and a cohomotopy class
$[\lambda] \in [X, S^{2k+1}]$ represented by a smooth map 
$\lambda : X \rightarrow S^{2k+1}$, we define the odd higher twisted K-theory of $X$,
twisted by the cohomotopy class $[\lambda]$ and denoted
$K^{\textbf{1}}_{\lambda}(X)$, to consist of homotopy
classes in $[X, L_{2k+1}]$ that lift the map $\lambda$. In other words, an element of 
$K^{\textbf{1}}_{\lambda}(X)$ is a homotopy class of a map $F : X \rightarrow L_{2k+1}$ such
that the following digram commutes

\begin{center}
\begin{tikzcd} 
{}
 & 
L_{2k+1}
\arrow[d, "\pi"] \\
 X  \arrow[r, swap, "\lambda"] \arrow[ur, "F"] &
 S^{2k+1}. 
\end{tikzcd}
\end{center}
\end{definition}

We remark that an equivalent way to define $K^{\textbf{1}}_{\lambda}(X)$ would be as homotopy classes of sections of the bundle 
$\Omega Fred_{\lambda}(X)$. This follows from the
fact that $\Omega Fred_{\lambda}(X) = \lambda^*L_{2k+1}$. Furthermore, as 
$Fred_{\mathcal{O}_{\infty}\otimes \mathcal{K}}$ is a smooth Banach manifold, it follows that
$\Omega Fred_{\mathcal{O}_{\infty}\otimes \mathcal{K}}$ is also a smooth Banach manifold. This implies the bundle $L_{2k+1}$ can be given the structure of a smooth Banach manifold, and hence the concept of a smooth map from $X$ to $L_{2k+1}$ is well defined.

We define the odd twisted Chern character
\begin{equation*}
Ch^{\textbf{1}}_{\lambda} : K^{odd}_{\lambda}(X) \rightarrow H^{\textbf{1}}_{\lambda^*(u_{2k+1})}(X)
\end{equation*}
by $Ch^{\textbf{1}}_{\lambda} ([F]) = [F^*Ch^{\textbf{1}}_{\lambda} ]$. We remind the reader that when we form the pull back differential form $F^*Ch^{\textbf{1}}_{\lambda}$, we always take a smooth representative $F$. 
The fact that this defines an element of the odd twisted cohomology group $H_{\lambda^*(u_{2k+1})}(X)$ follows by naturality and the fact that $Ch^{\textbf{1}}_{2k+1} \in  H^{\textbf{1}}_{e^*(\eta_{2k+1})}(L_{2k+1})$.

\section{Spherical T-duality}

\subsection{Historical introduction}
We begin with a historical introduction.

T-duality for pairs consisting of a circle bundle, together with a degree 3 H-flux, 
was originally studied in detail in \cite{BEM04a, BEM04b}
with contributions by several others later.  In string theory, T-dual pairs are distinct compactification 
manifolds that cannot be distinguished by any experiment, which is the notion of isomorphism relevant in physics.  
This equivalence in physics implies the isomorphisms of a number of other mathematical structures, such as Courant 
algebroids \cite{cavalcanti}, generalized complex structures \cite{cavalcanti} and twisted K-theory \cite{BEM04a}, see also 
\cite{BS,RR88}. It turns out that all of these structures are physically relevant.

Recently we initiated the study of spherical T-duality 
for principal $ SU(2)$-bundles in \cite{BEM14} and for non-principal $SU(2)$-bundles in  \cite{BEM142}, together with further insights in \cite{BEM18}. 
Let $P$ be a principal $ SU(2)$-bundle over $M$ and $H$ a 7-cocycle on $P$,
\begin{equation}\label{SU(2)1}
\begin{CD}
 SU(2) @>>> \,  P \\
&& @V \pi VV \\
&& M \end{CD}
\end{equation}

Principal $ SU(2)$-bundles over a compact oriented four dimensional manifold $M$ are classified 
by $H^4(M;\ZZ) \cong \ZZ$ via the 2nd Chern class $c_2(P)$. This can be seen using the well known isomorphism,
$H^4(M;\ZZ) \cong [M, S^4]\cong \ZZ$ and noting that there is a canonical principal $ SU(2)$-bundle 
$P\to S^4$, known as the Hopf bundle, whose 2nd Chern class is the generator of $ H^4(S^4;\ZZ)\cong \ZZ$.  
The orientation of $M$ and $ SU(2)$ imply that $\pi_*$ is a canonical isomorphism $ H^7(P;\ZZ)\cong H^4(M;\ZZ)\cong\Z$.  
The dual principal $ SU(2)$-bundle,
\begin{equation}\label{SU(2)1}
\begin{CD}
 SU(2) @>>> \,  \widehat P \\
&& @V \widehat\pi VV \\
&& M \end{CD}
\end{equation}
is defined by $c_2(\widehat P)=\pi_*H$ while the dual 7-cocycle $\widehat H\in H^7(\widehat P)$ satisfies $c_2(P) = 
{\widehat \pi}_*{\widehat H}$ by the isomorphism $\widehat\pi_* \colon  H^7(\widehat P;\ZZ)\cong H^4(M;\ZZ)\cong\Z$.  
We proved that this spherical T-duality map induces degree-shifting isomorphisms between the real and integral twisted 
cohomologies of $P$ and $\widehat P$ and also between the 7-twisted K-theories.

In  \cite{BEM142}, we extended spherical T-duality 
to (oriented) non-principal $ SU(2)$-bundles. 
While principal $SU(2)$-bundles correspond to unit sphere bundles
of quaternionic line bundles,  (oriented) non-principal $ SU(2)$-bundles
correspond to unit sphere bundles
of rank $4$ oriented real Riemannian vector bundles.
A striking new phenomenon in the non-principal case is 
that when the base $M$ is a compact oriented simply-connected 4 dimensional manifold, and given 
an oriented non-principal $ SU(2)$-bundle $E$ with $H$ a 7-cocycle on $E$,
then for each integer, we will show that there is an {\em infinite} lattice of spherical T-duals with 7-cocycle flux over $M$,
in stark contrast to the case of principal $ SU(2)$-bundles as described above.  
One reason is because in the non-principal
bundle case, the Euler class does not determine (oriented) non-principal $ SU(2)$-bundles when the base is as above, 
but in addition the 2nd Stiefel-Whitney
class and the Pontryagin class are also needed for the classification. However, it 
is only the Euler class (and its transgression) that is needed in the Gysin sequence and also to prove the isomorphisms
of 7-twisted integral cohomologies, and in addition the 2nd Stiefel-Whitney class is needed to prove the isomorphisms
of 7-twisted K-theories.

\subsection{Spherical T-duality computations}

Let $M$ be a compact, oriented manifold of dimension $2n$. 
Since in high dimensions, the classification up to isomorphism of $S^{2n-1}$ bundles over $M$ by cohomology 
is complicated, or unknown, we only consider the Euler characteristic of such bundles. In general 
there may be several bundles (over a fixed base $M$) that have the same Euler characteristic, but they all have the same cohomology
groups and the same K-theory groups, as we will show.

\subsection{Spherical T-dual pairs}
 
 Here we define spherical T-dual pairs over $M$ as above. Suppose that $n=3$ or $n>4$.
Consider the sphere bundle $Z $ over $M$ with structure group $SO(2n)$
 \begin{equation}\label{Spherebdle}
\begin{CD}
S^{2n-1} @>>> \,  Z  \\
&& @V\pi VV \\
&& M \end{CD}
\end{equation}
Then it has an Euler class $e(Z ) \in H^{2n}(M, \ZZ)$ which is an even integer $2k$ (\cite{Wals}).
Let $H $ be a degree $(4n-1)$-form over $Z $ such that $[H ] \in H^{4n-1}(Z , \ZZ)$ is the even integer $2l$. 
Then $(Z , H )$ has a spherical T-dual $(\widehat Z , \widehat H )$ defined as follows.  
  \begin{equation}\label{Spherebdle}
\begin{CD}
S^{2n-1} @>>> \,  \widehat Z  \\
&& @V \widehat\pi VV \\
&& M \end{CD}
\end{equation}
It has Euler class $e(\widehat Z ) = \pi_*(H )\in H^{2n}(M, \ZZ)$ that is equal to $2l$. 
Since $ \widehat\pi_* : H^{4n-1}(\widehat Z, \ZZ )\longrightarrow H^{2n}(M, \ZZ) $ is an isomorphism, we define $\widehat H
 = {\widehat\pi_*}^{-1}  \pi_*(H)$.

To summarise, $(Z , H )$ and $(\widehat Z , \widehat H )$ are defined to be {\em spherical T-dual pairs} over $M$.\\

Now we finish by defining spherical T-dual pairs over $M$ when $n=4$ i.e. an $8$ dimensional manifold as above. 
Consider the special sphere bundle $Z $ over $M$ with structure group $SO(8)$
 \begin{equation}\label{Spherebdle}
\begin{CD}
S^{7} @>>> \,  Z  \\
&& @V\pi VV \\
&& M\end{CD}
\end{equation}
Then the Euler class $e(Z ) \in H^{8}(M, \ZZ)$ can be any  integer $k$, since there is a Hopf bundle over $S^8$.
Let $H $ be a $15$-form over $Z $ such that $[H ] \in H^{15}(Z , \ZZ)$ is the integer $l$. 
Then $(Z , H )$ has a spherical T-dual $(\widehat Z , \widehat H )$ defined as follows.  
  \begin{equation}\label{Spherebdle}
\begin{CD}
S^{7} @>>> \,  \widehat Z  \\
&& @V \widehat\pi VV \\
&& M\end{CD}
\end{equation}
where  the Euler class $e(\widehat Z ) = \pi_*(H )\in H^4(M, \ZZ)$  is equal to $l$. 
Since $ \widehat\pi_* : H^{15}(\widehat Z, \ZZ )\longrightarrow H^{8}(M, \ZZ) $ is an isomorphism, we define $\widehat H  
 = {\widehat\pi_*}^{-1}  \pi_*(H )$ which equals $k$.

To summarise, $(Z , H )$ and $(\widehat Z , \widehat H )$ are {\em spherical T-dual pairs} over $M$.


\subsection{Gysin sequence and the cohomology of $Z $ and $\widehat Z  $}\label{gysin_cohom} 
Since  $Z $ is a sphere bundle over 
$M$ for $n=3$ or for $n>4$, with Euler class $e(Z)=2k$, there are Gysin 
sequences that will compute the cohomology of $Z$. Assume here that $k\ne 0$. The Gysin sequence is
$$
\to H^j(M, \ZZ) \stackrel{\pi^*}{\to} H^j(Z , \ZZ) \stackrel{\pi_*}{\to} H^{j-(2n-1)}(M, \ZZ)  \stackrel{\cup 2k}{\to} H^{j+1}(M, \ZZ) \to
H^{j+1}(Z , \ZZ)  \stackrel{\pi_*}{\to} H^{j+2-2n}(M, \ZZ) \to
$$ 
It follows that\\ 
\begin{eqnarray*}
H^j(Z , \ZZ) &\cong& H^j(M, \ZZ),  \quad 0\le j < 2n-1, i.e.\\
H^0(Z , \ZZ) &=&\ZZ, \,\, H^j(Z , \ZZ)=H^j(M, \ZZ) \quad \text{ for}\quad 0\le j \le 2n-1\\
H^{2n}(Z , \ZZ) &=&\ZZ_{2k}\\
H^j(Z , \ZZ)&=&H^{j+1-2n}(M, \ZZ), \quad 2n<j<4n-2\\
H^{4n-1}(Z , \ZZ)&=&\ZZ.\\
H^{odd}(Z , \ZZ) &\cong& \ZZ \bigoplus_{j=1}^{7} H^{j}(M,\ZZ).
\end{eqnarray*}
 \begin{eqnarray}
H^{even}(Z , \ZZ)&\cong&   \ZZ \oplus \ZZ_{2k} \bigoplus_{j=1}^{2n-1} H^{j}(M,\ZZ)  \\
H^{odd}(Z , \ZZ) &\cong&  \ZZ \bigoplus_{j=1}^{2n-1} H^{j}(M,\ZZ).
\end{eqnarray}
When $k=0$, then we get that 
 \begin{eqnarray}
H^{even}(Z , \ZZ)&\cong&   \ZZ \oplus \ZZ\bigoplus_{j=1}^{2n-1} H^{j}(M,\ZZ)  \\
H^{odd}(Z , \ZZ) &\cong&  \ZZ \bigoplus_{j=1}^{2n-1} H^{j}(M,\ZZ).
\end{eqnarray}

Similarly for $S^7$-bundles $Z $ over $M^8$ with Euler class $k$, repeating the calculations above, we see that for $k\ne 0$,

 \begin{eqnarray}
 H^{even}(Z , \ZZ) &\cong& \ZZ \oplus \ZZ_{k} \bigoplus_{j=1}^{7} H^{j}(M,\ZZ)\\
H^{odd}(Z , \ZZ) &\cong& \ZZ \bigoplus_{j=1}^{7} H^{j}(M,\ZZ).
\end{eqnarray}
When $k=0$ we get that 
 \begin{eqnarray}
 H^{even}(Z , \ZZ) &\cong& \ZZ \oplus \ZZ \bigoplus_{j=1}^{7} H^{j}(M,\ZZ)\\
H^{odd}(Z , \ZZ) &\cong& \ZZ \bigoplus_{j=1}^{7} H^{j}(M,\ZZ).
\end{eqnarray}

\subsection{Isomorphism of higher twisted cohomologies}\label{twisted_cohom}

We  use the calculations above to now compute the higher twisted cohomology groups $H^{even/odd}_H(Z, \ZZ) $, computed via the secondary operation 
$\cup H $ acting on $H^{\bullet}(Z , \ZZ) $. Note that $(\cup H)^2=0$ for degree reasons.

$\cup H  : H^0(Z , \ZZ) =\ZZ \longrightarrow H^{4n-1}(Z , \ZZ)=\ZZ$ is multiplication by $2l$, and since $l\ne0$, the nullspace is zero and the image is $2l\ZZ$. 
It follows that the twisted cohomology groups are 
\begin{eqnarray}
H^{even}_H(Z , \ZZ) &\cong&  \ZZ_{2k}\bigoplus_{j=1}^{2n-1} H^{j}(M,\ZZ)\nonumber\\
H^{odd}_H(Z , \ZZ) =&\cong&  \ZZ_{2l}\bigoplus_{j=1}^{2n-1} H^{j}(M,\ZZ).
\end{eqnarray}

Noting that  $(Z , H )$ and $(\widehat Z , \widehat H )$ are {\em spherical T-dual pairs} over $M$,
where $e(\widehat Z)  = 2l $ and $\widehat H = 2k$, we see that 
\begin{eqnarray}
H^{even}_{\widehat H}(\widehat Z , \ZZ) &\cong&  \ZZ_{2l}\bigoplus_{j=1}^{2n-1} H^{j}(M,\ZZ)\nonumber\\
H^{odd}_{\widehat H}(\widehat Z , \ZZ) &\cong&  \ZZ_{2k}\bigoplus_{j=1}^{2n-1} H^{j}(M,\ZZ).
\end{eqnarray}


We conclude that the {\em spherical T-duality isomorphisms} hold:
 \begin{eqnarray}
H^{even}_H(Z , \ZZ) &\cong&  H^{odd}_{\widehat H}(\widehat Z , \ZZ))\nonumber\\
H^{odd}_H(Z , \ZZ) &\cong&  H^{even}_{\widehat H}(\widehat Z , \ZZ).
\end{eqnarray}
The same is true when either $k=0$ or $l=0$.


For $S^7$-bundles $Z $ over $M=M^8$ with Euler class $k$ and degree 15-flux $H $ equal to $l$, we can calculate analogous to the above that
\\
$H^{even}_H(Z , \ZZ) =  \ZZ_{k}\bigoplus_{j=1}^{7} H^{j}(M,\ZZ)$ and $H^{odd}_H(Z , \ZZ) = \ZZ_{l}\bigoplus_{j=1}^{7} H^{j}(M,\ZZ)$ giving the {\em spherical T-duality isomorphisms}
 \begin{eqnarray}
H^{even}_H(Z , \ZZ) &\cong&  H^{odd}_{\widehat H}(\widehat Z , \ZZ))\nonumber\\
H^{odd}_H(Z , H , \ZZ) &\cong&  H^{even}_{\widehat H}(\widehat Z , \ZZ).
\end{eqnarray}

\subsection{K-theory of $Z$ and $\widehat{Z}$ when the base $M$ is torsion free}

In this subsection, we will prove that the even K-theory of $Z$ is isomorphic to
the even cohomology of $Z$, and that the odd K-theory of $Z$ is isomorphic to
the odd cohomology of $Z$, under the added assumption that the base manifold $M$ is torsion free.
Similar results hold for the $T$-dual 
$\widehat{Z}$.

As we will be applying the Gysin sequence, which depends on the $\ZZ$ grading of the $K$-theory groups, we will have to consider the $\ZZ$ graded $K$-groups
$K^{j}(\cdot)$ for $j \in \ZZ$. By Bott periodicity we know that these groups are 2 periodic. Therefore, we will let $K^{even}(\cdot)$ denote the even $K$-groups and $K^{odd}(\cdot)$ the odd $K$-groups. While this notation is slightly different to the one employed in section \ref{htk_ch}, it is employed so as to not cause confusion with the $K^0$ and $K^1$ terms that arise in the Gysin sequence.
In the case of cohomology, we follow the notation used in the previous subsection. That is, 
$H^{even}(\cdot;\ZZ)$ will denote the direct sum of all even cohomology groups, and $H^{odd}(\cdot;\ZZ)$ the direct sum of all the odd cohomology groups.

Using the Gysin sequence in $K$-theory, we obtain

\begin{center}
\begin{tikzcd} 
\cdots \arrow[r] & K^{j-2n}(M) \arrow[r, "\phi"] & K^j(M) 
\arrow[r, "\pi^*"] & K^j(Z) \arrow[r, "\pi_*"] \arrow[d, phantom, ""{coordinate, name=Z}] & K^{j-(2n-1)}(M) \arrow[dllll, "\phi", rounded corners, to path={ -- ([xshift=2ex]\tikztostart.east) |- (Z) [near end]\tikztonodes -| ([xshift=-2ex]\tikztotarget.west) -- (\tikztotarget)}] \\ 
K^{j+1}(M) \arrow[r, "\pi^*"] & K^{j+1}(Z) \arrow[r, "\pi_*"] 
& K^{j+(2n-1)}(M) \arrow[r] & \cdots\cdots
\end{tikzcd}
\end{center}

The map $\phi$ is the $K$-theoretic version of cupping with the Euler class, and is defined on the even $K$-groups by $\phi(x) = rk(x)k\cdot Th$, with $Th$ being the
Thom class and $rk$ being the rank function. 
On the odd $K$-groups 
$\phi \equiv 0$.

As we are assuming $M$ is torsion free, it follows that
$K^{even}(M) \cong H^{even}(M;\ZZ)$. This follows from a simple calculation using the Atiyah-Hirzebruch spectral sequence, and the fact that all differentials in this spectral sequence are torsion operators, see \cite{AD}. Cupping with the Euler class is a map from $H^0(M;\ZZ) \cong \ZZ$ to $H^{2n}(M;\ZZ) \cong \ZZ$, given by multiplication by $2k$. Under the isomorphism 
$K^{even}(M) \cong H^{even}(M;\ZZ)$, the map $\phi$ has image 
\[2k\ZZ \subseteq \ZZ \cong H^{2n}(M;\ZZ) \subseteq H^{even}(M;\ZZ).\]

We further observe that when $j$ is even, the kernel of 
$\phi$ is the reduced $K$-group $\widetilde{K}^{j-2n}(M)$. The fact that 
$K^{even}(M) \cong H^{even}(M;\ZZ)$ then gives 
$\ker\phi = \widetilde{K}^{j-2n}(M) \cong \bigoplus_{j=1}^nH^{2j}(M;\ZZ)$.

We start by analysing the even $K$-groups of $Z$.
When $j = 0$, we can extract out the following sequence using the fact that
$\phi \equiv 0$.

\begin{center}
\begin{tikzcd} 
0 \arrow[r] & K^{-1}(M) \arrow[r, "\pi^*"] & K^{-1}(Z) \arrow[r, "\pi_*"] 
& K^{-2n}(M) 
\arrow[r, "\phi"] & K^0(M) \arrow[r, "\pi^*"] \arrow[d, phantom, ""{coordinate, name=Z}] & K^0(Z) \arrow[dllll, "\pi_*", rounded corners, to path={ -- ([xshift=2ex]\tikztostart.east) |- (Z) [near end]\tikztonodes -| ([xshift=-2ex]\tikztotarget.west) -- (\tikztotarget)}] \\ 
{} & K^{-(2n-1)}(M) \arrow[r, "\phi"] & 0 
& {}  & {}
\end{tikzcd}
\end{center}

From this sequence we can extract two short exact sequences. The first exact sequence being

\begin{center}
\begin{tikzcd} 
0 \arrow[r] & \ker\pi^* = \im\phi \arrow[r] & K^{0}(M) 
\arrow[r, "\pi^*"] & \im\pi^* \arrow[r, "\phi"] & 0
\end{tikzcd}
\end{center}

which gives $\im\pi^* \cong \frac{K^0(M)}{\im\phi} \cong \frac{K^0(M)}{2k\ZZ}$.

As $K^0(M) \cong H^{even}(M; \ZZ)$, and $\im\phi$ can be identified with
$2k\ZZ \subseteq H^{2n}(M;\ZZ)$, we see that 
\[\im\pi^* \cong \bigoplus_{j = 0}^{n-1}H^{2j}(M;\ZZ) \oplus \ZZ_{2k}.\]

The second exact sequence we obtain is
\begin{center}
\begin{tikzcd} 
0 \arrow[r] & \ker\pi_* = \im\pi^* \arrow[r] & K^{0}(Z) 
\arrow[r, "\pi_*"] & K^{-(2n-1)}(M) \arrow[r, "\phi"] & 0
\end{tikzcd}
\end{center}

The torsion free condition on $M$ implies $K^{2n-1}(M) \cong H^{odd}(M;\ZZ)$, which is a torsion free abelian group. Hence, the above short exact sequence splits and, on applying Bott periodicity, we obtain 
\[K^{even}(Z) \cong 
\bigoplus_{j = 0}^{n-1}H^{2j}(M;\ZZ) \oplus \ZZ_{2k} \oplus H^{odd}(M;\ZZ).\]

The odd $K$-groups of $Z$ can be analysed in a similar manner. Letting 
$j = 1$ in the $K$-theoretic Gysin sequence, and using the fact that
$\phi \equiv 0$ on the odd $K$-groups, gives the following sequence

\begin{center}
\begin{tikzcd} 
0 \arrow[r] & K^{1}(M) \arrow[r, "\pi^*"] & K^{1}(Z) \arrow[r, "\pi_*"] 
& K^{2-2n}(M) 
\arrow[r, "\phi"] & K^2(M) \arrow[r, "\pi^*"] \arrow[d, phantom, ""{coordinate, name=Z}] & K^2(Z) \arrow[dllll, "\pi_*", rounded corners, to path={ -- ([xshift=2ex]\tikztostart.east) |- (Z) [near end]\tikztonodes -| ([xshift=-2ex]\tikztotarget.west) -- (\tikztotarget)}] \\ 
{} & K^{3-2n}(M) \arrow[r, "\phi"] & 0 
& {}  & {}
\end{tikzcd}
\end{center}

From this exact sequence we can extract the short exact sequence

\begin{center}
\begin{tikzcd} 
0 \arrow[r] & K^1(M) \arrow[r, "\pi^*"] & K^{1}(Z) 
\arrow[r, "\pi_*"] & \im\pi_* = \ker\phi \arrow[r] & 0
\end{tikzcd}
\end{center}

As $\ker\phi \cong \bigoplus_{j=1}^nH^{2j}(M;\ZZ)$ is a free abelian group, the 
above short exact sequence splits. Applying Bott periodicity, we 
thus obtain 
\[K^{odd}(Z) \cong K^1(M) \oplus  \bigoplus_{j=1}^nH^{2j}(M;\ZZ) 
\cong H^{odd}(M;\ZZ) \oplus \bigoplus_{j=1}^nH^{2j}(M;\ZZ).\] 
These observations allow us to prove the following theorem.

\begin{theorem}\label{k-theory_Z}
$K^{even/odd}(Z) \cong H^{even/odd}(Z;\ZZ)$.
\end{theorem}

\begin{proof}
The above discussion shows $K^{even}(Z) \cong 
\bigoplus_{j = 0}^{n-1}H^{2j}(M;\ZZ) \oplus Z/ 2k\ZZ \oplus H^{odd}(M;\ZZ)
\cong H^{even}(Z;\ZZ)$. This last isomorphism coming from the Gysin sequence computation in section \ref{gysin_cohom}.

Similarly, using the cohomology Gysin sequence computation for $H^{odd}(Z;\ZZ)$, carried out in section \ref{gysin_cohom}, and the above computation for 
$K^{odd}$, proves the odd case.
\end{proof}
A similar result holds for the $T$-dual $\widehat{Z}$. As the proof is exactly analogous to what was done above for $Z$, we simply state the final theorem.

\begin{theorem}
$K^{even/odd}(\widehat{Z}) \cong H^{even/odd}(\widehat{Z};\ZZ)$.
\end{theorem}

\subsection{Isomorphism of higher Twisted $K$-theories when the base $M$ is torsion free}

In this subsection we compute the higher twisted K-theory of $Z$ and $\widehat{Z}$, and show that they are isomorphic after a degree shift. This should be compared with similar results obtained for the twisted cohomologies in section \ref{twisted_cohom}. As in the previous subsection, we will be assuming that $M$ is torsion free.
We remind the reader that $Z$ has Euler class $e(Z) \in H^{2n}(M;\ZZ)$ corresponding to the integer $2k$.
We also fixed a top degree cohomology class $H \in H^{4n-1}(Z;\ZZ)$, that corresponds to the integer $2l$. The spherical $T$-dual $\widehat{Z}$ has Euler class $e(\widehat{Z}) = 2l$. Furthermore, we have fixed $\widehat{H} \in 
H^{4n-1}(\widehat{Z};\ZZ)$ given by $\widehat{H} = 2k$. We will be computing the twisted K-theory of $Z$, twisted by the class $H$ and denoted by 
$K^{\textbf{0}}_H$ and $K^{\textbf{1}}_H$, as well as the twisted K-theory of $\widehat{Z}$, twisted by the class
$\widehat{H}$. It is these twisted $K$-theory groups that will be shown to be isomorphic after a degree shift.  
We will do this via the Atiyah-Hirzebruch spectral sequence for twisted $K$-theory.

\includegraphics[width=15cm, height=15cm]{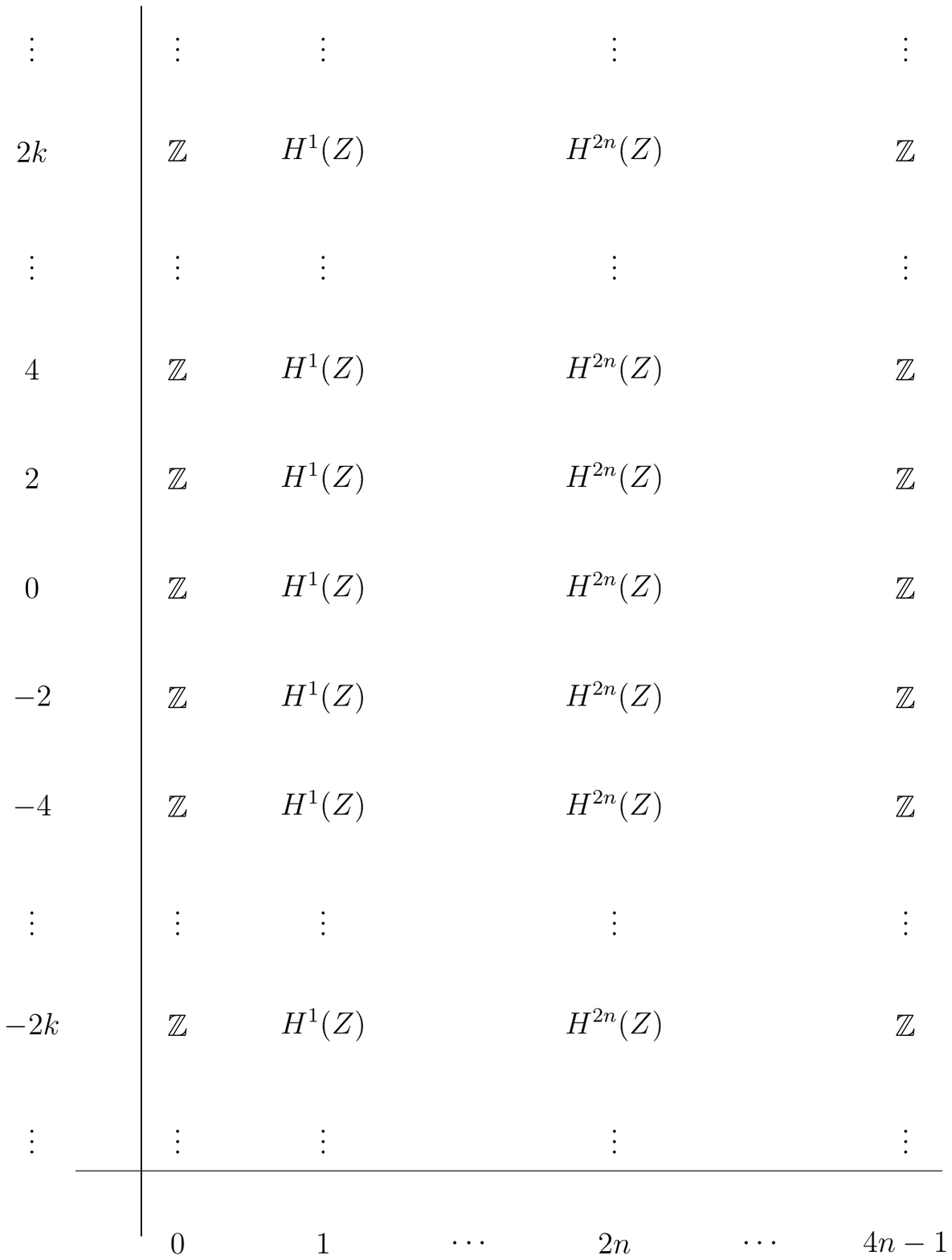}

Therefore, the $E_2$-page of the spectral sequence (as in the figure) takes the following form.
The $E_2$-page of the Atiyah-Hirzebruch spectral sequence for $K^*_H$ is given
by $E_2^{p,q} = H^p(Z; K^{q}(pt)$, where $K$ denotes ordinary complex 
$K$-theory. As $K^{odd}(pt) = 0$ and $K^{even}(pt) = \ZZ$, we see that
$E_2^{p, 2j} = H^p(Z;\ZZ)$ and $E_2^{p,2j+1} = 0$ for all $j \in \ZZ$.


From the spectral sequence we can see that the even differentials all vanish.
In the untwisted case, (i.e. when $H = 0$) it is known that the odd 
differentials are all torsion operators, see \cite{AD}. It follows that the odd 
differentials  $d_j$ for $j < 4n-1$ are all torsion operators, as these are equal to the differentials in the untwisted case.  
As 
$Z$ only has torsion in $H^{2n}(Z;\ZZ)$, we see that the potential 
non-zero differentials are $d_3, d_5, \cdots ,d_{2n-1}$ and $d_{4n-1}$.

\begin{proposition}\label{odd_diff_vanish}
The differentials $d_3, d_5, \cdots ,d_{2n-1}$ are all zero.
\end{proposition}
\begin{proof}
The projection $\pi : Z \rightarrow M$ induces a map from the Atiyah-Hirzebruch 
spectral sequence of $K^*(M)$ to the Atiyah-Hirzebruch spectral sequence of 
$K^*(P)$.

As the $E_3$-page of the Atiyah-Hirzebruch spectral sequence is the $E_2$-page, it follows that the induced map on the $E_3$-page, induced by $\pi$, is $\pi^*$ and this commutes with $d_3 : H^{2n-3}(Z;\ZZ) \rightarrow H^{2n}(Z;\ZZ)$. We have already seen in section \ref{gysin_cohom} that 
$\pi^* : H^{j}(M;\ZZ) \rightarrow H^{j}(Z;\ZZ)$ is an isomorphism for $j < 2n$. If we take $x \in H^{2n-3}(Z;\ZZ)$ let $y \in H^{2n-3}(M;\ZZ)$ satisfy 
$\pi^*(y) = x$. We then see that $d_3(x)= d_3\circ\pi^*(y) = 
\pi^*\circ d_3(y) = 0$. This latter equality following from the fact that
$H^{2n}(M;\ZZ) \cong \ZZ$ and $d_3$ being a torsion operator. Thus we have proved that $d_3 \equiv 0$.

The general case proceeds by induction. Assume that the differentials 
$d_{2j - 1} \equiv 0$ for all $2 \leq j \leq i-1 < n$. We will prove that
$d_{2i-1} \equiv 0$.

The induction hypothesis tells us that the induced map on the $E_{2i-1}$-page, 
induced by $\pi$, is still $\pi^*$ and that this map commutes with
$d_{2i-1}$. Given $x \in H^{2n-(2i-1)}(Z;\ZZ)$, and using the fact that
$\pi^* : H^{2n-(2i-1)}(M;\ZZ) \rightarrow H^{2n-(2i-1)}(Z;\ZZ)$ is an 
isomorphism, let $y \in H^{2n-(2i-1)}(M;\ZZ)$ such that $\pi^*(y) = x$.
We then have $d_{2i-1}(x) = d_{2i-1}\circ\pi^*(y) = \pi^*\circ d_{2i-1}(y) = 0$, where again we have used the fact that $d_{2i-1}$ is a torsion operator and hence must vanish on $H^{2n-(2i-1)}(M;\ZZ)$.

We have proved that the differentials $d_3, d_5, \cdots ,d_{2n-1}$ are all zero in the Atiyah-Hirzebruch spectral sequence for the $K$-theory groups $K^*(Z)$. However, these differentials are also the differentials in the twisted Atiyah-Hirzebruch spectral sequence for $K^*_H(Z)$ in that range (the difference occurs for the differentials after $d_{4n-3}$).
Thus the theorem has been proved.

\end{proof}

\begin{proposition}\label{diff_cup}
The differential $d_{4n-1} = \cup H = \cup 2l$.
\end{proposition}
\begin{proof}
The cohomology class $H \in H^{4n-1}(Z;\ZZ)$ is represented by a map
$\widetilde{H} : Z \rightarrow K(\ZZ, 4n-1)$, such that 
$\widetilde{H}^*(\eta) = H$, where $\eta$ is a generator of 
$H^{4n-1}(K(\ZZ, 4n-1);\ZZ)$.
As $Z$ has dimension $4n-1$, we can apply the cellular approximation theorem, and assume the image 
of $\widetilde{H}$ is contained in the $4n-1$-skeleton of $K(\ZZ, 4n-1)$. The $4n-1$-skeleton of $K(\ZZ, 4n-1)$ is $S^{4n-1}$, and therefore we have
$\widetilde{H} : Z \rightarrow S^{4n-1}$ with a generator 
$\eta \in H^{4n-1}(S^{4n-1};\ZZ)$ satisfying $\widetilde{H}^*(\eta) = H$.

The map $\widetilde{H}$ induces a map of the associated twisted Atiyah-Hirzebruch spectral sequences of $Z$ and $S^{4n-1}$ respectively. Therefore, it suffices to prove the proposition for $S^{4n-1}$ with twist given by the generator $\eta$. This latter case follows from the Mayer-Vietoris argument in \cite{BS}.

\end{proof}

Propositions \ref{odd_diff_vanish} and \ref{diff_cup} imply that the only
non-zero differential in the Atiyah-Hirzebruch spectral sequence for
$Z$ is $d_{4n-1} : H^0(Z;\ZZ) \cong \ZZ \rightarrow H^{4n-1}(Z;\ZZ) \cong \ZZ$.

\includegraphics[width=17cm, height=18cm]{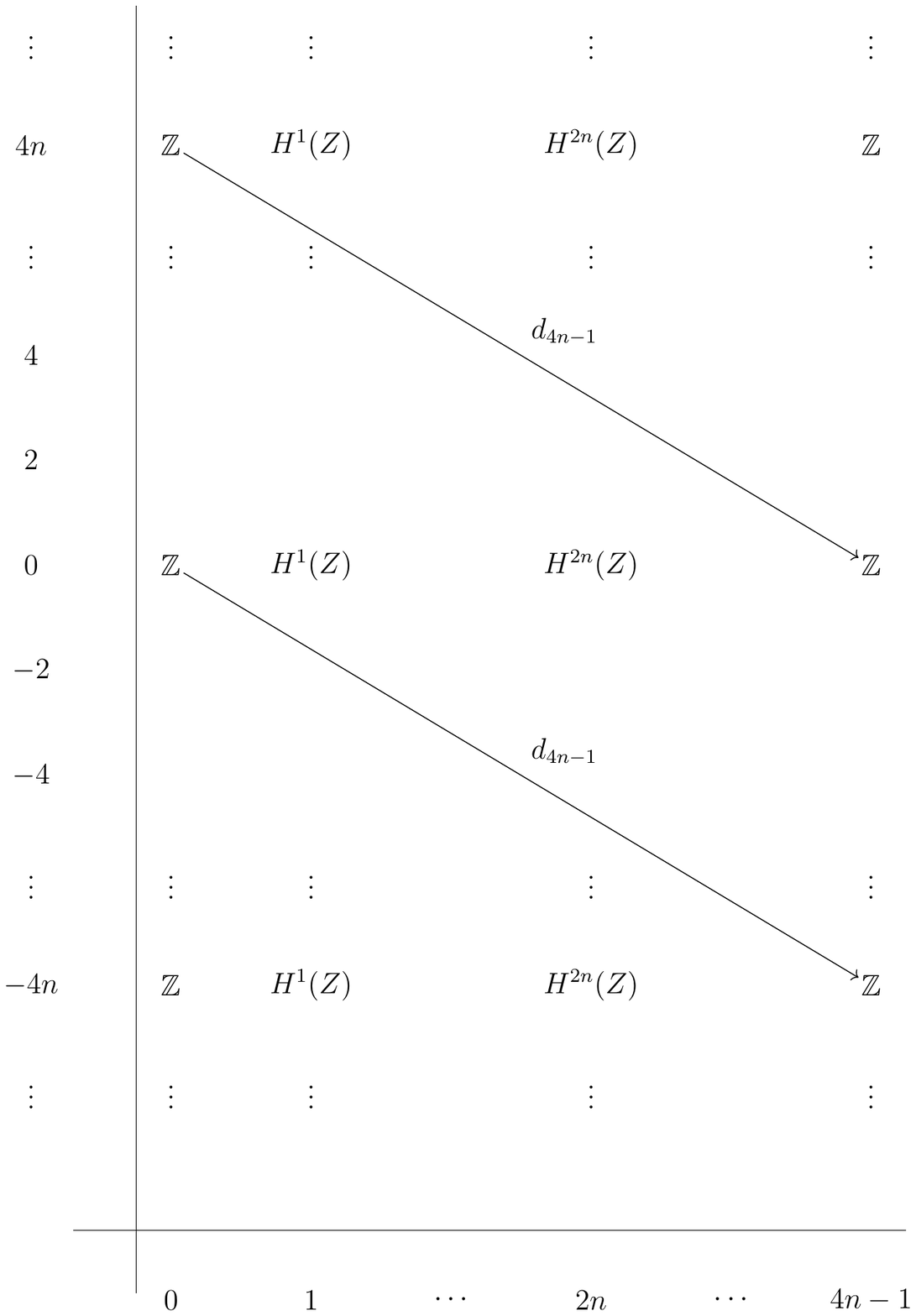}

We thus find that the $E_{\infty}$-page is given in the following way:

\includegraphics[width=17cm, height=18cm]{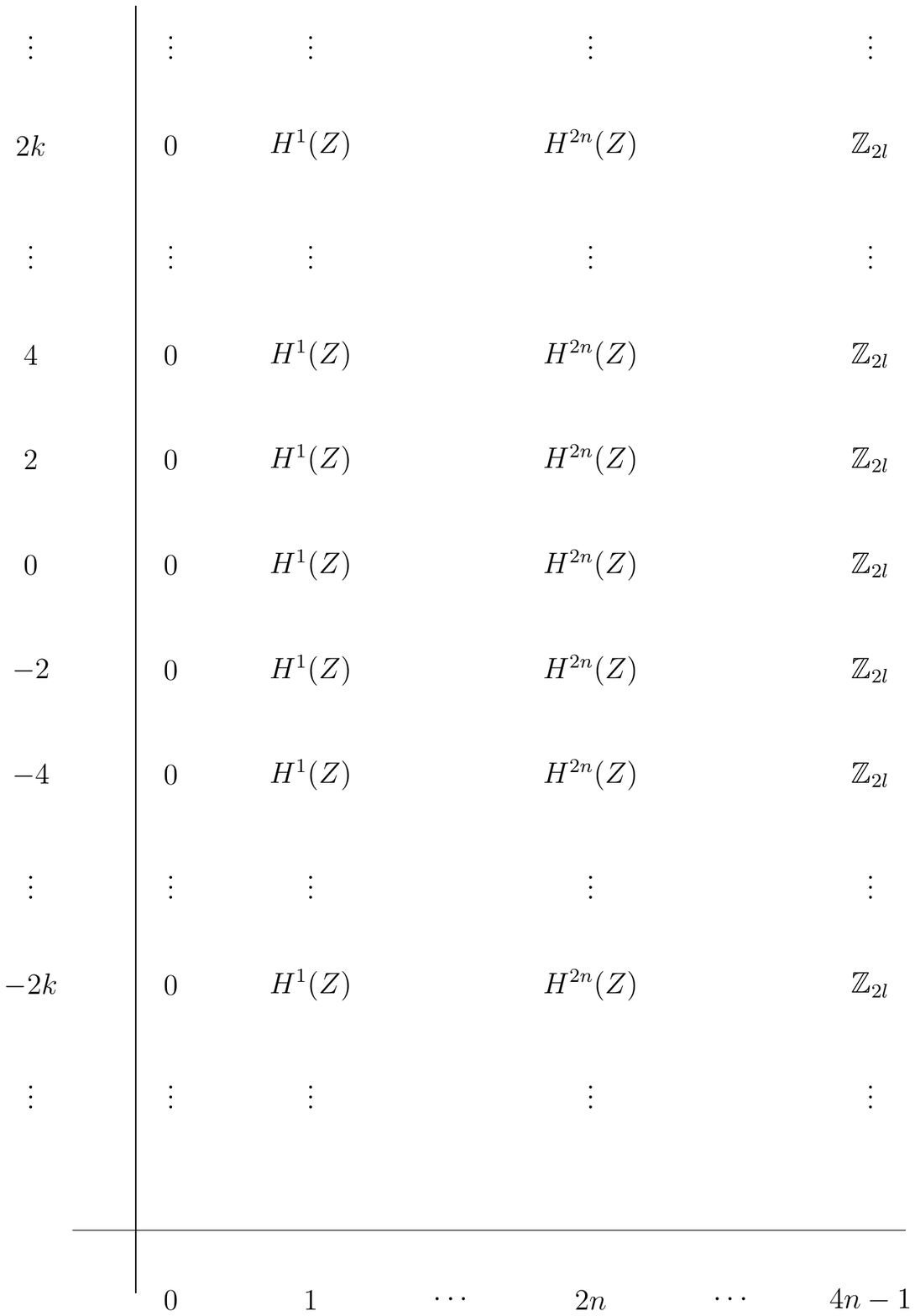}

From the $E_{\infty}$-page of the spectral sequence we can compute the
twisted K-theory groups $K_H^*(Z)$. We do this via the filtration of the twisted K-theory groups. We remind the reader how this proceeds.

Given the sphere bundle $Z$, of dimension $4n-1$, we fix a CW-complex structure on $Z$. Note that the CW-complex structure on $Z$ will have cells in dimensions $k \leq 4n-1$, we will denote the $k^{th}$-skeleton of $Z$ by $Z_k$.
We then have the following filtration of $K^n_H(Z)$, 
$n \in \ZZ$,
\begin{equation}\label{filt_ss}
K^n_H(Z) \supseteq \mathcal{F}^0 \supseteq \mathcal{F}^1 \supseteq \ldots
\supseteq \mathcal{F}^{4n-2} \supseteq \mathcal{F}^{4n-1} = \{0\}
\end{equation}
where $\mathcal{F}^i := \ker(K_H^n(Z) \rightarrow K_H^n(Z_i))$, the map
$K_H^n(Z) \rightarrow K_H^n(Z_i)$ being induced by the inclusion
$Z_i \hookrightarrow Z$.

The convergence of the spectral sequence implies that 
$E_{\infty}^{p,q} = \frac{\mathcal{F}^{p-1}}{\mathcal{F}^p}$, where
$p + q = n$.  We refer the reader to the final section for a more detailed description of this spectral sequence.

The above $E_{\infty}$-page for $K^*_H(Z)$ differs from the untwisted complex K-theory $K^*(Z)$ in only the $0$ and $4n-1$ lines. As the $0$-line consists entirely of zeros, it follows that $K^{\textbf{0}}_H(Z) \cong \widetilde{K}(Z)$, the reduced $K$-theory of $Z$. Applying theorem \ref{k-theory_Z} we obtain the following proposition.

\begin{proposition}\label{twisted_k_even}
$K^{\textbf{0}}_H(Z) \cong \bigoplus_{i = 1, i \neq n}^{2n-1}H^{2i}(Z;\ZZ) 
\oplus \ZZ_{2k} \cong H^{even}_H(Z)$. 
\end{proposition}

The following proposition computes the odd twisted $K$-theory.

\begin{proposition}\label{twisted_k_odd}
$K_H^{\textbf{1}}(Z) \cong \bigoplus_{i=0}^{2n-2}H^{2i+1}(Z;\ZZ)\oplus \ZZ_{2l} 
\cong H^{odd}_H(Z;\ZZ)$
\end{proposition}
\begin{proof}
We observe that since $E_{\infty}^{p,q} = 0$ when $q$ is odd, we have that
$\mathcal{F}^0 = K_H^{\textbf{1}}(Z)$ and $\mathcal{F}^{2i} = \mathcal{F}^{2i-1}$ for
$i > 0$. In particular, our filtration of $K_H^{\textbf{1}}(Z)$ can be written as
\begin{equation*}
K_H^{\textbf{1}}(Z) \supseteq \mathcal{F}^1 \supseteq \mathcal{F}^3
\supseteq \cdots \supseteq \mathcal{F}^{4n-3} \supseteq \mathcal{F}^{4n-1}
= \{0\}
\end{equation*}
where $\frac{K_H^{\textbf{1}}(Z)}{\mathcal{F}^1} \cong H^1(Z;\ZZ)$,
$\frac{\mathcal{F}^{2i-1}}{\mathcal{F}^{2i+1}} \cong H^{2i+1}(Z;\ZZ)$ for
$1 \leq i \leq n-2$, and 
$\frac{\mathcal{F}^{4n-3}}{\mathcal{F}^{4n-1}} \cong \ZZ_{2l}$. As 
$\mathcal{F}^{4n-1} = \{0\}$, we immediately obtain 
$\mathcal{F}^{4n-3} = \ZZ_{2l}$.

The filtration determines short exact sequences
\begin{equation*}
0 \rightarrow \mathcal{F}^{2i+1} \rightarrow \mathcal{F}^{2i-1}
\rightarrow \frac{\mathcal{F}^{2i-1}}{\mathcal{F}^{2i+1}} \rightarrow 0
\end{equation*}
for each $0 \leq i \leq n-1$.

In the case that $i = n-1$, the short exact sequence tells us that
$\mathcal{F}^{4n-3} = \ZZ_{2l}$, which we already knew.

In the case that $0 \leq i < n-1$, we note that
the groups $H^{2i+1}(Z;\ZZ)$, for $0 \leq i \leq n-1$, are torsion free. 
This implies that the above short exact sequences split for all 
$0 \leq i < n-1$. 

These observations immediately imply 
\begin{equation*}
K_H^{\textbf{1}}(Z) \cong \bigoplus_{i=0}^{2n-2}H^{2i+1}(Z;\ZZ)\oplus \ZZ_{2l}
\end{equation*} 
which proves the first isomorphism in the statement of the proposition. The 
second isomorphism follows from the computations carried out in section 
\ref{twisted_cohom}.

\end{proof}

The arguments we gave, via a spectral sequence, to compute $K^*_H(Z)$
can be applied to compute $K^*_{\widehat{H}}(\widehat{Z})$. As the details are 
exact analogous to what we did above for $Z$, we simply state the 
results in the following proposition.

\begin{proposition}\label{twisted_k_dual}
We have the following isomorphisms
\begin{align*}
K^{\textbf{0}}_{\widehat{H}}(\widehat{Z}) &\cong 
\bigoplus_{i=1, i \neq n}^{2n-1}H^{2i}(\widehat{Z};\ZZ) \oplus \ZZ_{2l} \cong
H^{even}_{\widehat{H}}(\widehat{Z};\ZZ) \\
K^{\textbf{1}}_{\widehat{H}}(\widehat{Z}) &\cong 
\bigoplus_{i=0}^{2n-2}H^{2i}(\widehat{Z};\ZZ) \oplus \ZZ_{2k} 
\cong H^{odd}_{\widehat{H}}(\widehat{Z};\ZZ).
\end{align*}
\end{proposition}

From these computations we thus see that the spherical $T$-dual pairs $(Z, H)$ and 
$(\widehat{Z}, \widehat{H})$ have isomorphic twisted $K$-theories upto a degree shift.

\begin{theorem}\label{twisted_k_isos}
We have the following isomorphism
\begin{equation*}
K_H^{\textbf{0}/\textbf{1}}(Z) \cong K_{\widehat{H}}^{\textbf{1}/\textbf{0}}(\widehat{Z}).
\end{equation*}
\end{theorem}

\section{Spectral sequences}

To show that the Chern character is a real isomorphism from the twisted $K$-theory of a closed manifold $X$ to the twisted cohomology of $X$, in \cite{AS1} Atiyah and Segal consider two spectral sequences.  The first of these is a generalisation of the Atiyah-Hirzebruch spectral sequence \cite[Section 5]{AS1}, which computes the twisted $K$-theory of $X$, while the second computes the twisted cohomology of $X$.  We give in this section the correct generalisations of these spectral sequences to the higher twisted setting (noting that the second of these has already been covered in great detail in \cite{LLW}).

\subsection{The higher twisted Atiyah-Hirzebruch spectral sequence}

The Atiyah-Hirzebruch spectral sequence for $K$-theory was introduced for finite CW complexes $X$ (of which closed manifolds are examples) in the paper \cite{AH}.  As pointed out in \cite{AS1}, there is an analogous spectral sequence for twisted $K$-theory.

Our goal in this subsection is to show that there is an Atiyah-Hirzebruch spectral sequence for \emph{higher} twisted $K$-theory also, which generalises that in \cite{AS1}.  This spectral sequence will later be compared with that for higher twisted cohomology.  To construct the Atiyah-Hirzebruch spectral sequence, we must shift our attention from \emph{manifolds} to \emph{CW complexes}.  Moreover, we will work with \emph{general twists}, arising from classifying maps into $BAut(\cO_{\infty}\otimes\cK)$, which need not correspond to cohomology classes as outlined in the Preliminaries. 

\begin{definition}
	By a \textbf{pair with twist} we mean a triple $(X,A,\lambda)$, where $X$ is a finite CW complex, $A\subset X$ a CW subcomplex, and $\lambda:X\rightarrow BAut(\cO_{\infty}\otimes\cK)$ a twist over $X$.  Let $(X,A,\lambda_{X})$ and $(Y,B,\lambda_{Y})$ be pairs with twist.  A \textbf{morphism} $f:(X,A,\lambda_{X})\rightarrow(Y,B,\lambda_{Y})$ is a morphism $f$ of CW pairs such that $\lambda_{Y}\circ f = \lambda_{X}$.
\end{definition}

We remind the reader that, given a twist $\lambda:X\rightarrow BAut(\cO_{\infty}\otimes\cK)$ and a closed subspace $A\subset X$, then denoting by $F^{n}_{\lambda}$ the associated bundle over $X$ with fibre $\Omega^{n}Fred_{\cO_{\infty}\otimes\cK}$, the higher twisted $K$-theory group $K^{n}_{\lambda}(X,A)$ is the homotopy classes of sections of $F^{n}_{\lambda}$ that are equal to the identity outside of some compact set in $X-A$.  Clearly now if $f:(X,A,\lambda_{X})\rightarrow (Y,B,\lambda_{Y})$ is a morphism of pairs with twist, then $f^{*}F^{n}_{\lambda_{Y}} = F^{n}_{\lambda_{X}}$, so that the pullback by $f$ of any section of $F^{n}_{\lambda_{Y}}$ is a section of $F^{n}_{\lambda_{X}}$.  That is:

\begin{proposition}
	The assignment $(X,A,\lambda)\mapsto K^{n}_{\lambda}(X,A)$ is a functor from the category of pairs with twist to the category of abelian groups.\qed
\end{proposition}

As identified by Pennig \cite{Pennig}, one way of obtaining the six-term exact sequence in higher twisted $K$-theory is to use an identification with operator $K$-theory groups.  For the purpose of showing that our Chern character is a real isomorphism, we choose here to take a different, more topological approach inspired by Atiyah \cite{A67}.

Associated to the pair with twist $(X,A,\lambda)$, we have the sequence $(A,\emptyset,\lambda)\rightarrow(X,\emptyset,\lambda)\rightarrow(X,A,\lambda)$ of pairs with twist, which by functoriality induces the exact sequence
\[
K^{n}_{\lambda}(X,A)\xrightarrow{\ext_{*}}K^{n}_{\lambda}(X)\xrightarrow{\rest_{*}}K^{n}_{\lambda}(A)
\]
of $K$-theory groups.  Here $\ext_{*}$ is induced by extension of sections by the identity, and $\rest_{*}$ is induced by restriction of sections.  To obtain the corresponding long exact sequnece, we must give a connecting homomorphism $\partial:K^{n}_{\lambda}(A)\rightarrow K^{n+1}_{\lambda}(X,A)$.  We follow the definition given by Atiyah for ordinary $K$-theory \cite[p. 71-72]{A67}, with some necessary modifications from Pennig \cite{Pennig} to deal with the fact that twists do not extend to one-point compactifications.

Following Pennig, define $c(X,A):=X\sqcup(A\times[0,1))/\sim$, where $(a,0)\in A\times[0,1)$ is glued to $a\in X$.  We call this the \emph{open cone} on $A$.  The twist $\lambda$ over $X$ extends trivially to a twist $c\lambda$ over $c(X,A)$, and via a neighbourhood retract of $c(X,A)$ onto $X$ we have
\[
K^{n+1}_{\lambda}(X,A) = K^{n+1}_{c\lambda}(c(X,A)).
\]
Associated then to the map $m:(c(X,A),\emptyset,c\lambda)\rightarrow(c(X,A),X,c\lambda)$ of pairs with twist, we obtain a map
\[
m^{*}:K^{n+1}_{c\lambda}(c(X,A),X)\rightarrow K^{n+1}_{c\lambda}(c(X,A)) = K^{n+1}_{\lambda}(X,A).
\]
Now, the space $c(X,A)-X$ is simply the product $A\times(0,1)$ (which we call the \emph{open suspension} of $A$).  Thus if $\sigma$ is a section of $F^{n+1}_{c\lambda}$ over $c(X,A)-X$ which is equal to the identity outside of some compact set, we may define a new section $\tilde{\sigma}:A\rightarrow F^{n+2}_{\lambda}$ by the formula
\[
\tilde{\sigma}(a)(e^{2\pi it}):=\sigma(a,t),\hspace{7mm}a\in A,\,t\in S^{1}.
\]
Note that well-definedness of $\tilde{\sigma}(a)$ as a function on the circle follows from the fact that $\sigma$ is equal to the identity outside of some compact subset of $A\times(0,1)$.  The map $\sigma\mapsto\tilde{\sigma}$ descends to a well-defined map on homotopy classes, whose composite with the Bott periodicity isomorphism $K^{n+2}_{\lambda}(A)\cong K^{n}_{\lambda}(A)$ therefore defines a map
\[
\theta:K^{n+1}_{c\lambda}(c(X,A),X)\rightarrow K^{n}_{\lambda}(A).
\]
We thus obtain our connecting homomorphism
\begin{equation}\label{connecting}
\partial_{n}:=m^{*}\circ\theta^{-1}:K^{n}_{\lambda}(A)\rightarrow K^{n+1}_{\lambda}(X,A).
\end{equation}
Applying Bott periodicity and using similar arguments to those of Atiyah \cite[Proposition 2.4.4]{A67}, we therefore have the six-term sequence defined in topological terms.

\begin{theorem}\label{sixterm}
	Let $(X,A,\lambda)$ be a pair with twist.  Then associated to the sequence
	\[
	(A,\emptyset,\lambda)\rightarrow(X,\emptyset,\lambda)\rightarrow(X,A,\lambda)
	\]
	of pairs with twist we have the exact sequence
	\begin{center}
		\begin{tikzcd}
		K^{\mathbf{0}}_{\lambda}(X,A)\ar[r,"\ext_{*}"] & K^{\mathbf{0}}_{\lambda}(X)\ar[r,"\rest_{*}"] & K^{\mathbf{0}}_{\lambda}(A) \ar[d,"\partial_{\mathbf{0}}"] \\ K^{\mathbf{1}}_{\lambda}(A) \ar[u,"\partial_{\mathbf{1}}"] & K^{\mathbf{1}}_{\lambda}(X) \ar[l,"\rest_{*}"] & K^{\mathbf{1}}_{\lambda}(X,A) \ar[l,"\ext_{*}"]
		\end{tikzcd}
	\end{center}
	of twisted $K$-groups.\qed
\end{theorem}

We use Theorem \ref{sixterm} to give an Atiyah-Hirzebruch spectral sequence computing $K^{*}_{\lambda}(X)$ as follows.  Denote the $p$-skeleton of $X$ by $X^{p}$.  Each inclusion $X^{p}\hookrightarrow X$ induces by a homomorphism $K^{*}_{\lambda}(X)\rightarrow K^{*}_{\lambda}(X^{p})$ on higher twisted $K$-theory induced by restricting sections of $F_{\lambda}$ to $X^{p}$.  We thus filter the higher twisted $K$-theory groups $K^{*}_{\lambda}(X)$ of $X$ by the skeleta of $X$:
\[
\FF^{p}K^{*}_{\lambda}(X):=\ker\big(K^{*}_{\lambda}(X)\rightarrow K^{*}_{\lambda}(X^{p-1})\big) = \im\big(K^{*}_{\lambda}(X,X^{p-1})\rightarrow K^{*}_{\lambda}(X)\big),\hspace{7mm}p>0,
\]
taking in particular $\FF^{0}K^{*}_{\lambda}(X) = K^{*}_{\lambda}(X)$.  We have the following analogue of \cite[Theorem 2.1]{AH}.

\begin{theorem}[Higher twisted Atiyah-Hirzebruch spectral sequence]\label{AH}
	Let $X$ be a finite CW complex, equipped with a higher twist $P$.  Then there is a spectral sequence $\{E^{p}_{r}(K)\}$ with
	\[
	E_{2}^{p}(K)\cong H^{p}(X,\ZZ),\hspace{7mm}E_{\infty}^{p}(K) =\bigoplus_{*\in\ZZ_{2}} \FF^{p}K^{p\bmod2+*}_{\lambda}(X)/\FF^{p+1}K^{p\bmod2+*}_{\lambda}(X),
	\]
	which computes $K^{*}_{\lambda}(X)$.
\end{theorem}

\begin{proof}
	For each $p>0$, we have the sequence
	\[
	(X^{p-1},\emptyset,\lambda)\rightarrow(X^{p},\emptyset,\lambda)\rightarrow(X^{p},X^{p-1},\lambda)
	\]
	of pairs with twist, and thus by Theorem \ref{sixterm} we obtain the long exact sequence
	\begin{equation}\label{les}
	\begin{tikzcd}
	K^{\mathbf{0}}_{\lambda}(X^{p},X^{p-1})\ar[r,"\ext_{*}"] & K^{\mathbf{0}}_{\lambda}(X^{p}) \ar[r,"\rest_{*}"] & K^{\mathbf{0}}_{\lambda}(X^{p-1}) \ar[d, "\partial_{\mathbf{0}}"] \\ K^{\mathbf{1}}_{\lambda}(X^{p-1}) \ar[u,"\partial_{\mathbf{1}}"] & K^{\mathbf{1}}_{\lambda}(X^{p}) \ar[l, "\rest_{*}"] & K^{\mathbf{1}}_{\lambda}(X^{p},X^{p-1}) \ar[l,"\ext_{*}"]
	\end{tikzcd}
	\end{equation}
	of higher twisted $K$-theory groups.  We define $A^{p,*}_{1}(K):=K^{*}_{\lambda}(X^{p})$, $E^{p,*}_{1}(K):=K^{*}_{\lambda}(X^{p},X^{p-1})$, and denote the associated bigraded groups by
	\[
	A_{1}(K):=\bigoplus_{p\in\ZZ,\,*\in\ZZ_{2}}A^{p,*}_{1}(K),\hspace{7mm}E_{1}(K):=\bigoplus_{p\in\ZZ,\,*\in\ZZ_{2}}E^{p,*}_{1}(K).
	\]
	From the maps in the exact sequence \eqref{les}, we obtain:
	\begin{enumerate}
		\item $\alpha_{1}(K):A_{1}(K)\rightarrow A_{1}(K)$ of bi-degree $(-1,0)$ induced by $\rest_{*}$,
		\item $\beta_{1}(K):A_{1}(K)\rightarrow E_{1}(K)$ of bi-degree $(0,0)$ induced by $\ext_{*}$, and
		\item $\gamma_{1}(K):E_{1}(K)\rightarrow A_{1}(K)$ of bi-degree $(1,1)$ induced by $\partial_{*}$,
	\end{enumerate}
	and exactness of \eqref{les} therefore guarantees that 
	\begin{equation}\label{ec}
	\begin{tikzcd}
	A_{1}(K) \ar[rr,"\alpha_{1}(K)"] & & A_{1}(K) \ar[dl,"\beta_{1}(K)"] \\ & E_{1}(K) \ar[ul, "\gamma_{1}(K)"] &
	\end{tikzcd}
	\end{equation}
	is an exact couple.  The Atiyah-Hirzebruch spectral sequence $\{E_{r}^{p,*}(K)\}$ is now obtained from the derived couples of the exact couple \eqref{ec}, with differentials $d_{r}:E^{p,*}_{r}(K)\rightarrow E_{r}^{p+r,*+1}(K)$ of bi-degree $(r,1)$.  Due to the fact that $X^{p} = X$ for $p$ sufficiently large, this spectral sequence stabilises with $E_{\infty}(K)$ term
	\[
	E_{\infty}^{p,*}(K) = \FF^{p}K^{p\bmod2+*}_{\lambda}(X)/\FF^{p+1}K^{p\bmod2+*}_{\lambda}(X),
	\]
	hence converges to $K^{*}_{\lambda}(X)$ as claimed.
	
	Finally let us show that the $E_{2}(K)$ term is as claimed.  We have that
	\[
	E_{1}^{p,*}(K) = K^{*}_{\lambda}(X^{p},X^{p-1})
	\]
	is the homotopy classes of sections of the $\Omega^{*}Fred_{\cO_{\infty}\otimes\cK}$-bundle $F^{*}_{\lambda}$ over $X^{p}-X^{p-1}$ which are the identity outside some compact set.  Now $X^{p}- X^{p-1}$ consists precisely of the interiors of the $p$-cells of $X$, and since the interior of any $p$-cell is contractible $F^{*}_{\lambda}|_{X^{p}-X^{p-1}}$ is trivialisable.  Therefore $K^{*}_{\lambda}(X^{p},X^{p-1})$ coincides with the ordinary $K$-theory $K^{*}(X^{p},X^{p-1})$ of the pair $(X^{p},X^{p-1})$, and the ordinary relative Chern character restricts thereon to an isomorphism $K^{\mathbf{0}}(X^{p},X^{p-1})\cong H^{even}(X^{p},X^{p-1};\ZZ)$ with relative singular cohomology with coefficients in $\ZZ$.  Naturality of the Chern character as a transformation of (generalised) cohomology functors then says that $d_{1}(K)$ is the cellular coboundary, giving the result.
\end{proof}

\subsection{Singular de Rham theory}

To show that the Chern character for higher twisted $K$-theory is a real isomorphism, we must consider twisted cohomology for CW complexes.  Since the Chern character form we constructed earlier on the $Fred_{\cO_{\infty}\otimes\cK}$-bundle $Y_{2k+1}$ over $S^{2k+1}$ is a differential form, in constructing the Chern character for a CW complex with twist $(X,\lambda)$, we must use a cohomology theory which is simultaneously compatible with CW complexes and with differential forms.  As pointed out by Atiyah-Segal \cite{AS1}, \emph{singular de Rham} theory is perfect for this role.  We review this theory for the reader's convenience here in the convenient category of \emph{diffeological spaces}, whose definition we now recall following \cite[Definition 1.5]{IZ}.

\begin{definition}
	Let $X$ be a set.  A \textbf{parametrisation} is a function $\varphi:U\rightarrow X$, where $U$ is some open subset in some finite-dimensional Euclidean space.  A \textbf{diffeology} on $X$ is a set D of parametrisations for which the following hold.
	\begin{enumerate}
		\item Every constant parametrisation is contained in D.
		\item If $\varphi:U\rightarrow X$ is a parametrisation such that every $u\in U$ admits an open neighbourhood $V\subset U$ for which $\varphi|_{V}:V\rightarrow X$ is an element of D, then $\varphi$ is an element of D.
		\item If $\varphi:U\rightarrow X$ is an element of D, $V$ is any open subset in some finite-dimensional Euclidean space and $f:V\rightarrow U$ is a smooth function, then $\varphi\circ f:V\rightarrow X$ is contained in D.
	\end{enumerate}
	A set with a diffeology is called a \textbf{diffeological space}, and the elements of the diffeology are called \textbf{plots}.  A function $f:X\rightarrow Y$ of diffeological spaces is said to be \textbf{smooth} if, for every plot $\varphi:U\rightarrow X$ of $X$, the composite $f\circ\varphi:U\rightarrow Y$ is a plot of $Y$.  Any subset $S$ of a diffeological space $X$ can be equipped with the \textbf{subspace diffeology}, for which a parametrisation $\varphi:U\rightarrow S$ is a plot if and only if its composition with the inclusion $S\hookrightarrow X$ is a plot of $X$.
\end{definition}

Diffeological spaces were originally introduced by Soriau \cite{soriau}, who was motivated by problems in quantisation.  They are vast generalisations of manifolds, and are used to apply tools of differential geometry to spaces that are not well-behaved by traditional standards, for instance orbifolds \cite{IKZ} and, increasingly, leaf spaces of foliations \cite{HMVSC, GV, mac1, mac2}.  Here we choose to work in the diffeological category out of convenience - it allows us to use a single language to speak of the singular de Rham theory of both topological spaces and manifolds as we will now describe.

Any topological space can be equipped with a diffeology, called its \emph{continuous diffeology} \cite[Section 3]{CSW}, whose plots are all those parametrisations which are continuous.  Smooth maps between topological spaces $X$ and $Y$ with the continuous diffeology are precisely the continuous maps $X\rightarrow Y$.  Moreover, any manifold has a standard diffeology whose plots are all those parametrisations which are smooth in the manifold sense, and a diffeologically smooth map of manifolds is the same thing as a smooth map in the usual sense.  In particular, every geometric simplex
\[
\Delta^{n}:=\{(x_{1},\dots,x_{n})\in\RR^{n}:0\leq x_{1}\leq\cdots\leq x_{n}\leq 1\}
\]
is a diffeological subspace of Euclidean space.  The singular de Rham complex of a diffeological space is defined as follows.

\begin{definition}\label{singdr}
	Let $X$ be a diffeological space, and let $k\geq0$.  A \textbf{singular simplex} is a diffeologically smooth map $\sigma:\Delta^{l}\rightarrow X$.  A \textbf{singular $k$-form} $\omega$ on $X$ is a map which assigns to each singular simplex $\sigma:\Delta^{l}\rightarrow X$ a differential $k$-form $\omega(\sigma)\in\Omega^{k}(\Delta^{l})$, subject to the constraint that the restriction of $\omega(\sigma)$ to the $i^{th}$ face of the geometric simplex $\Delta^{l}$ is equal to $\omega(\epsilon_{i}(\sigma))$, where $\epsilon_{i}$ is the $i^{th}$ face map.  The set $\Omega^{k}_{sing}(X)$ of all singular $k$-forms is a real vector space under componentwise operations:
	\[
	(\alpha\omega+\beta\kappa)(\sigma):=\alpha\omega(\sigma)+\beta\kappa(\sigma),\hspace{7mm}\alpha,\beta\in\RR,\,\omega,\kappa\in\Omega^{k}_{sing}(X).
	\]
	The direct sum $\Omega^{*}_{sing}(X):=\bigoplus_{k\geq0}\Omega^{k}_{sing}(X)$ is a differential graded algebra under pointwise operations:
	\[
	(d\omega)(\sigma):=d\omega(\sigma),\hspace{7mm}(\omega\wedge\kappa)(\sigma):=\omega(\sigma)\wedge\kappa(\sigma),\,\hspace{7mm}\omega,\kappa\in\Omega^{*}(X).
	\]
	The differential graded algebra $\Omega^{*}_{sing}(X)$ is called the \textbf{singular de Rham complex} of $X$, and its cohomology denoted $H^{*}_{sing-dR}(X)$.
\end{definition}

For manifolds, singular de Rham cohomology is related to ordinary de Rham cohomology by a pullback map.  For a proof we refer the reader to \cite[Theorem 2.4]{kuri}.

\begin{theorem}\label{iso2}
	Let $X$ be a manifold.  Then the map $\alpha:\Omega^{*}(X)\rightarrow\Omega^{*}_{sing}(X)$ defined by
	\[
	\alpha(\omega)(\sigma):=\sigma^{*}\omega
	\]
	is a quasi-isomorphism.\qed
\end{theorem}

There is a diffeological de Rham theorem \cite[Corollary 2.5]{kuri}, which relates singular de Rham theory to singular theory.

\begin{theorem}\label{iso}
	Let $X$ be a diffeological space, and denote by $C^{*}(X)$ the complex of smooth singular cochains of $X$.  Then the integration map $\int:\Omega^{*}_{sing}(X)\rightarrow C^{*}(X)$ defined by
	\[
	\bigg(\int\omega\bigg)(\sigma:\Delta^{l}\rightarrow X):=\int_{\Delta^{l}}\omega(\sigma),\hspace{7mm}\omega\in\Omega^{*}_{sing}(X),
	\]
	is a cochain map, and is moreover a quasi-isomorphism, inducing an isomorphism $H^{*}_{sing-dR}(X)\cong H^{*}(X;\RR)$ on cohomology.\qed
\end{theorem}

In particular, the singular de Rham cohomology of a topological space is isomorphic to the ordinary singular cohomology with real coefficients.  The singular de Rham complexes of diffeological spaces are functorial under smooth maps.

\begin{proposition}\label{compat}
	Let $f:X\rightarrow Y$ be a smooth map of diffeological spaces.  Then the formula
	\[
	(f^{*}\omega)(\sigma):=\omega(f\circ\sigma),\hspace{7mm}\omega\in\Omega^{*}_{sing}(Y)
	\]
	defines a map $f^{*}:\Omega^{*}_{sing}(Y)\rightarrow\Omega^{*}_{sing}(X)$ of cochain complexes.  In particular, if $X$ and $Y$ are manifolds, then the diagram
	\begin{center}
		\begin{tikzcd}
		\Omega^{*}(Y) \ar[r,"\alpha"] \ar[d,"g^{*}"] & \Omega^{*}_{sing}(Y) \ar[d,"g^{*}"] \\ \Omega^{*}(X) \ar[r,"\alpha"] & \Omega^{*}_{sing}(X)
		\end{tikzcd}
	\end{center}
	commutes.
\end{proposition}

\begin{proof}
	The first part is clear.  For the second, simply fix $\omega\in\Omega^{*}(Y)$ and observe that
	\[
	g^{*}(\alpha(\omega))(\sigma) = \alpha(\omega)(g\circ\sigma) = (g\circ\sigma)^{*}\omega = \sigma^{*}(g^{*}\omega) = \alpha(g^{*}\omega)(\sigma)
	\]
	for any smooth singular simplex $\sigma$ of $X$.
\end{proof}

We have the following crucial lemma, which relies on the extendability property \cite[Lemma 9.4]{GM} of the differential forms on geometric simplices.

\begin{lemma}\label{technical}
	Let $f:X\rightarrow Y$ be an injective smooth map of diffeological spaces.  Then $f^{*}:\Omega^{*}_{sing}(Y)\rightarrow\Omega^{*}_{sing}(X)$ is surjective.
\end{lemma}

\begin{proof}
	Fix $\omega\in\Omega^{*}_{sing}(X)$.  We must construct a singular form $\tilde{\omega}\in\Omega^{*}_{sing}(Y)$ such that $f^{*}\tilde{\omega} = \omega$.  Let $\tilde{\sigma}:\Delta^{n}\rightarrow Y$ be a smooth singular simplex.  If $\tilde{\sigma} = f\circ\sigma$ for some singular simplex $\sigma$ of $X$, then by injectivity of $f$, $\sigma$ is uniquely characterised by this property and we simply define
	\[
	\tilde{\omega}(\tilde{\sigma}):=\omega(f\circ\sigma).
	\]
	In general, while $\tilde{\sigma}$ may not itself be of the form $f\circ\sigma$, it may have faces which \emph{are} of this form.  To take care of this subtlety we must in general construct $\tilde{\omega}$ inductively.
	
	Let $v_{0},\dots,v_{n}$ be the vertices of $\Delta^{n}$.  We define
	\[
	\tilde{\omega}([v_{i}]):=\begin{cases}
	\omega(\sigma_{i})&\text{ if $\tilde{\sigma}|_{v_{i}} = f\circ\sigma_{i}$ for some 0-simplex $\sigma_{i}$ in $X$}\\ 0&\text{ otherwise}.
	\end{cases}
	\]
	Given a multi-index $I_{k}:=(i_{1},\dots,i_{k})$, denote by $[v_{I_{k}}]$ the corresponding $(k-1)$-face
	\[
	[v_{I_{k}}]:=[v_{i_{1}},\dots v_{i_{k}}],
	\]
	and for $1\leq j\leq k$ denote by $[\epsilon_{j}v_{I_{k}}]$ the $j^{th}$ face of $[v_{I_{k}}]$.  Suppose inductively that we have forms $\tilde{\omega}([v_{I_{k}}])$ defined on all the $(k-1)$-faces of $\tilde{\sigma}$ which agree on intersections, and given any $k$-face $[v_{I_{k+1}}]$ of $\tilde{\sigma}$ denote by
	\[
	\ext\big(\tilde{\omega}([\epsilon_{1}v_{I_{k+1}}]),\dots,\tilde{\omega}([\epsilon_{k+1}v_{I_{k+1}}])\big)
	\]
	the form on $v_{I_{k+1}}$ extending the $\tilde{\omega}_{\epsilon_{i}v_{I_{k+1}}}$ obtained via \cite[Lemma 9.4]{GM}.  Then we define
	\[
	\tilde{\omega}([v_{I_{k+1}}]):=\begin{cases}
	\omega(\sigma_{I_{k+1}})&\text{ if $\tilde{\sigma}|_{[v_{I_{k+1}}]} = f\circ\sigma_{I_{k+1}}$, for $\sigma_{I_{k+1}}:\Delta^{k}\rightarrow X$}\\\ext(\tilde{\omega}([\epsilon_{1}v_{I_{k+1}}]),\dots,\tilde{\omega}([\epsilon_{k+1}v_{I_{k+1}}])&\text{ otherwise}
	\end{cases}
	\]
	to obtain a family of forms $\tilde{\omega}([v_{I_{k+1}}])$ associated to the $k$-faces of $\tilde{\sigma}$.  After $n$ steps, we obtain the desired form $\tilde{\omega}(\tilde{\sigma})\in\Omega^{*}(\Delta^{n})$, and if in particular $\tilde{\sigma} = f\circ\sigma$ for some simplex $\sigma:\Delta^{n}\rightarrow X$, then $\tilde{\omega}(\tilde{\sigma}) = \omega(\sigma)$ by construction.
	
	Doing this construction for each singular simplex $\tilde{\sigma}$ of $Y$, we obtain a singular form $\tilde{\omega}\in\Omega^{*}_{sing}(Y)$, and if $\sigma:\Delta^{n}\rightarrow X$ is any singular simplex then
	\[
	(f^{*}\tilde{\omega})(\sigma) = \tilde{\omega}(f\circ\sigma) = \omega(\sigma).
	\]
	Thus $f^{*}$ is surjective as claimed.
\end{proof}

Theorem \ref{iso} implies a similar isomorphism for relative cohomology groups, which we now define.

\begin{definition}
	Let $X$ be a diffeological space, and let $A\subset X$ be a diffeological subspace.  We denote by $\Omega^{*}_{sing}(X,A)$ the differential graded subalgebra of $\Omega^{*}_{sing}(X)$ consisting of singular forms $\omega$ such that $\omega(\sigma) = 0$ whenever $\sigma:\Delta^{l}\rightarrow A$ is a singular simplex of $A$.  The cohomology of $\Omega^{*}_{sing}(X,A)$ is denoted $H^{*}_{sing-dR}(X,A)$.
\end{definition}

If $X$ is any diffeological space and $A$ any diffeological subspace, then the restriction map from $\Omega^{*}_{sing}(X)$ to $\Omega^{*}_{sing}(A)$ is surjective by Lemma \ref{technical}, and lifts of singular forms in $A$ to singular forms in $X$ are unique up to elements of $\Omega^{*}_{sing}(X,A)$.  Thus $\Omega^{*}_{sing}(A)$ is isomorphic to the quotient $\Omega^{*}_{sing}(X)/\Omega^{*}_{sing}(X,A)$.  Therefore we have an exact sequence
\[
0\rightarrow\Omega^{*}_{sing}(X,A)\xrightarrow{\iota}\Omega^{*}_{sing}(X)\xrightarrow{\pi}\Omega^{*}_{sing}(A)\rightarrow 0
\]
of cochain complexes, where $\iota$ is the inclusion and where $\pi$ is restriction.

Now the integration map described in Theorem \ref{iso} sends $\Omega^{*}_{sing}(X,A)$ into the relative singular cochains $C^{*}(X,A;\RR)$ - that is, cochains which vanish on singular simplices which map into $A$.  We then have the commuting diagram
\begin{center}
	\begin{tikzcd}
	0 \ar[r] & \Omega^{*}_{sing}(X,A) \ar[d,"\int"] \ar[r,"\iota"] & \Omega^{*}_{sing}(X) \ar[d,"\int"] \ar[r,"\pi"] & \Omega^{*}_{sing}(A) \ar[r] \ar[d,"\int"] & 0 \\ 0 \ar[r] & C^{*}(X,A;\RR) \ar[r,"\iota"] & C^{*}(X;\RR) \ar[r,"\pi"] & C^{*}(A;\RR) \ar[r] & 0
	\end{tikzcd}
\end{center}
of cochain complexes.  Passing to the long exact sequences guaranteed by the zig zag lemma, we obtain commuting diagrams
\begin{center}
	\begin{tikzcd}
	H^{n}_{sing-dR}(X) \ar[r,"\pi_{*}"] \ar[d,"\int_{*}"] & H^{n}_{sing-dR}(A) \ar[r,"\partial_{*}"] \ar[d,"\int_{*}"] & H^{n+1}_{sing-dR}(X,A) \ar[r,"\iota_{*}"] \ar[d,"\int_{*}"] & H^{n+1}_{sing-dR}(X) \ar[r,"\pi_{*}"] \ar[d,"\int_{*}"] & H^{n+1}_{sing-dR}(A) \ar[d,"\int_{*}"] \\ H^{n}(X;\RR) \ar[r,"\pi_{*}"] & H^{n}(A;\RR) \ar[r,"\partial_{*}"] & H^{n+1}(X,A;\RR) \ar[r,"\iota_{*}"] & H^{n+1}(X;\RR) \ar[r,"\pi_{*}"] & H^{n+1}(A;\RR)
	\end{tikzcd}
\end{center}
for each $n$, with exact rows.  The five lemma together with Theorem \ref{iso} therefore proves the following.

\begin{proposition}
	For a diffeological space $X$ and any subspace $A$ of $X$, the integration map from singular de Rham forms to singular cochains induces an isomorphism $H^{*}_{sing-dR}(X,A)\cong H^{*}(X,A;\RR)$.\qed
\end{proposition}

We conclude this subsection by considering the relationship between the singular de Rham theories of a manifold when equipped with its standard manifold diffeology versus when it is equipped with its continuous diffeology.

\begin{proposition}\label{topsm}
	Let $X$ be a manifold, and denote by $X^{t}$ and $X^{m}$ the set $X$ regarded with its continuous diffeology and with its manifold diffeology respectively.  The identity map $\id:X^{m}\rightarrow X^{t}$ is smooth, and the induced map $\id^{*}:\Omega^{*}_{sing}(X^{t})\rightarrow\Omega^{*}_{sing}(X^{m})$ is a quasi-isomorphism.  Consequently, every closed element of $\Omega^{*}(X^{m})$ has a lift to a closed element of $\Omega^{*}_{sing}(X^{t})$.
\end{proposition}

\begin{proof}
	The first claim follows from commutativity of the diagram
	\begin{center}
		\begin{tikzcd}
		\Omega^{*}_{sing}(X^{t}) \ar[r,"\id^{*}"] \ar[d,"\int"] & \Omega^{*}_{sing}(X^{m}) \ar[d,"\int"] \\ C^{*}(X;\RR) \ar[r,"\id^{*}"] & C^{*}_{sm}(X;\RR),
		\end{tikzcd}
	\end{center}
	where $C^{*}_{sm}(X;\RR)$ denotes the cochain complex of singular cochains on \emph{smooth} singular simplices, and where the bottom $\id^{*}$ denotes restriction of cochains from continuous to smooth.  Since the bottom $\id^{*}$ is a quasi-isomorphism \cite[p. 205]{warner} and since both integration maps are quasi-isomorphisms by Theorem \ref{iso}, the top $\id^{*}$ is also a quasi-isomorphism.  The second claim then follows from injectivity of $\id$ together with Lemma \ref{technical}.  Indeed, if $\omega\in\Omega^{n}_{sing}(X^{m})$ is closed, then since $\id^{*}$ is a surjective quasi-isomorphism we can find closed $\kappa_{1}\in\Omega^{n}_{sing}(X^{t})$ and $\kappa_{2}\in\Omega^{n-1}_{sing}(X^{t})$ such that
	\[
	\id^{*}(\kappa_{1}) = \omega+d(\id^{*}\kappa_{2}).
	\]
	Then $\tilde{\omega}:=\kappa_{1}-d\kappa_{2}$ is a closed lift of $\omega$.
\end{proof}

\subsection{The higher twisted cohomology spectral sequence}

In contrast to the twisted cohomology spectral sequences considered in \cite{AS1} and \cite{LLW}, our twisted cohomology spectral sequence will be constructed in a very similar fashion to the Atiyah-Hirzebruch spectral sequence, in terms of filtrations by skeleta.  Let us fix for the entirety of this subsection a finite CW complex $X$, with a \emph{cohomological twist} $[H]\in H^{2k+1}(X;\ZZ)$ that will be represented by a singular $(2k+1)$-form $H\in\Omega^{2k+1}_{sing}(X)$.

As pointed out by Atiyah and Segal, since $H$ is an odd-degree singular form, the twisted differential
\[
D:=d-H\wedge
\]
squares to zero on the singular de Rham complex $\Omega^{*}_{sing}(X)$ of $X$, and preserves the $\ZZ_{2}$-grading thereon.  We can therefore consider the $\ZZ_{2}$-graded \emph{twisted} singular de Rham cohomology $H^{*}_{H}(X)$ of the complex $(\Omega^{*}_{sing}(X),D)$.  Similar remarks hold for the relative groups $(X,A)$ for any CW subcomplex $A$ of $X$, and we may consider the twisted relative groups $H^{*}_{H}(X,A)$.  The following theorem now follows from the zig zag lemma.

\begin{theorem}\label{sixtermcohom}
	Let $(X,A,H)$ be a CW pair with cohomological twist represented by a singular $(2k+1)$-form $H$.  Associated to the sequence $(A,\emptyset,H)\rightarrow(X,\emptyset,H)\rightarrow(X,A,H)$ of pairs with twist, we have the short exact sequnece
	\[
	0\rightarrow\Omega^{*}_{sing}(X,A)\xrightarrow{\iota}\Omega^{*}_{sing}(X)\xrightarrow{\pi}\Omega^{*}_{sing}(A)\rightarrow0
	\]
	of (twisted) differential graded algebras, giving rise to a six term exact sequence
	\begin{center}
		\begin{tikzcd}
		H^{\mathbf{0}}_{H}(X,A)\ar[r,"\iota_{*}"] & H^{\mathbf{0}}_{H}(X)\ar[r,"\pi_{*}"] & H^{\mathbf{0}}_{H}(X,A) \ar[d,"D_{\mathbf{0}}"] \\ H^{\mathbf{1}}_{H}(X,A) \ar[u,"D_{\mathbf{1}}"] & H^{\mathbf{1}}_{H}(X) \ar[l,"\pi_{*}"] & H^{\mathbf{1}}_{H}(X,A) \ar[l,"\iota_{*}"]
		\end{tikzcd}
	\end{center}
	of twisted cohomology groups.\qed
\end{theorem}

Note that the connecting homomorphism in the six term sequence of Theorem \ref{sixtermcohom} is induced by the twisted differential $D$ after identifying $\Omega^{*}_{sing}(A)$ with $\Omega^{*}_{sing}(X)/\Omega^{*}_{sing}(X,A)$ as in the previous subsection.

Consider now the filtration of $X$ by its skeleta $X^{p}$.  Exactly as for twisted $K$-theory, we can consider the corresponding filtration
\[
\FF^{p}H^{*}_{H}(X):=\ker\big(H^{*}_{H}(X)\rightarrow H^{*}_{H}(X^{p-1})\big) = \im\big(H^{*}_{H}(X,X^{p-1})\rightarrow H^{*}_{H}(X)\big)
\]
of twisted cohomology by the skeleta of $X$.  Now substituting $K^{*}$ for $H^{*}$, and $\ext_{*}$, $\rest_{*}$ and $\partial_{*}$ for $\iota_{*}$, $\pi_{*}$ and $D_{*}$ in Theorem \ref{AH}, we have the following theorem giving a spectral sequence which computes $H^{*}_{H}(X)$.

\begin{theorem}\label{specseqcohom}
	Let $(X,H)$ be a finite CW complex with twist represented by $H\in\Omega^{2k+1}_{sing}(X)$.  Then there is a spectral sequence $\{E^{p}_{r}(H)\}$ with
	\[
	E_{2}^{p}(H)\cong H^{p}(X,\ZZ),\hspace{7mm}E_{\infty}^{p}(H) =\bigoplus_{*\in\ZZ_{2}} \FF^{p}H^{p\bmod2+*}_{H}(X)/\FF^{p+1}H^{p\bmod2+*}_{H}(X),
	\]
	which computes $H^{*}_{H}(X)$.
\end{theorem}

\begin{proof}
	As mentioned, the proof of Theorem \ref{AH} applies almost without change.  Let us just be careful about specifying the $E_{1}$ page.  The $E_{1}$ page of this spectral sequence comes to us from an exact couple built out of the six term sequence of Theorem \ref{sixtermcohom} with $X$ replaced by $X^{p}$ and $A$ replaced by $X^{p-1}$.  Thus
	\[
	E_{1}^{p,*}(H) = H^{*}_{H}(X^{p},X^{p-1}).
	\]
	Now the action of the twisted differential on $\Omega^{*}(X^{p},X^{p-1})$ is determined by the restriction of $H$ to $X^{p}-X^{p-1}$, which is a disjoint union of contractible open $p$-balls.  Therefore $H$ is exact on $X^{p}-X^{p-1}$, and we may choose a transgression of $H$ so as to obtain an identification of $H^{*}_{H}(X^{p},X^{p-1})$ with the relative singular de Rham cohomology $H^{*}_{sing-dR}(X^{p},X^{p-1})$ (see \cite[Section 6, Remark (ii)]{AS1}), and hence with the ordinary singular cohomology $H^{*}(X^{p},X^{p-1};\RR)$ by the integration map of Theorem \ref{iso}.  That the $E_{2}$ page is as claimed will follow from our comparison of the spectral sequences $E_{r}(K)$ and $E_{r}(H)$ via the Chern character in the next subsection (see Theorem \ref{chiso}), so we do not carry out the computation here.
\end{proof}

\subsection{Comparison of the spectral sequences}

Recall from Section 2 that even and odd twisted Chern character forms $\Ch^{\mathbf{0}}_{2k+1}$ and $\Ch^{\mathbf{1}}_{2k+1}$ have been constructed on the bundles $Y_{2k+1}$ and $L_{2k+1}$ over $S^{2k+1}$, with fibres $Fred_{\cO_{\infty}\otimes\cK}$ and $\Omega Fred_{\cO_{\infty}\otimes\cK}$ respectively.  We require the following.

\begin{lemma}\label{pullbackid}
	The pullback $\id_{Y_{2k+1}}^{*}(Ch^{\mathbf{0}}_{2k+1})$ of $Ch^{\mathbf{0}}_{2k+1}\in\Omega^{*}(Y_{2k+1})$ by the identity section $\id_{Y_{2k+1}}:S^{2k+1}\rightarrow Y_{2k+1}$ is equal to zero.  Similarly, the pullback $\id_{L_{2k+1}}^{*}(Ch^{\mathbf{1}}_{2k+1})$ of $Ch^{\mathbf{1}}_{2k+1}\in\Omega^{*}(L_{2k+1})$ by the identity section $\id_{L_{2k+1}}:S^{2k+1}\rightarrow L_{2k+1}$ is zero.
\end{lemma}

\begin{proof}
	We prove the result only for $Ch^{\mathbf{0}}_{2k+1}$, as the odd case is similar.  About any point in $S^{2k+1}$, we can find an open neighbourhood $U$ and a trivialisation $Y_{2k+1}|_{U}\cong U\times Fred_{\cO_{\infty}\otimes\cK}$ with respect to which $Ch^{\mathbf{0}}_{2k+1}$ is just the pullback of a Chern character form on $Fred_{\cO_{\infty}\otimes\cK}$ by the projection $\pi_{2}:U\times Fred_{\cO_{\infty}\otimes\cK}\rightarrow Fred_{\cO_{\infty}\otimes\cK}$ onto the second factor.  We then have that $\pi_{2}\circ\id_{Y_{2k+1}}|_{U}$ is the constant map $U\rightarrow\{\id_{Fred_{\cO_{\infty}\otimes\cK}}\}$, giving the result.
\end{proof}

By a \emph{finite CW pair with cohomotopy twist} we mean a triple $(X,A,H)$, where $X$ is a finite CW complex, $A$ a CW subcomplex, and $H = \tilde{H}^{*}(\overline{\alpha(\eta)})\in\Omega^{2k+1}_{sing}(X)$ a singular de Rham form associated to a cohomotopy representative $\tilde{H}:X\rightarrow S^{2k+1}$.  Here $\eta\in\Omega^{2k+1}(S^{2k+1})$ is the form used to construct $Ch^{\mathbf{0}}_{2k+1}$ and $Ch^{\mathbf{1}}_{2k+1}$ as in Section 2, and $\overline{\alpha(\eta)}$ is an extension of $\eta$ to a continuous singular de Rham form as in Theorem \ref{iso2} and Proposition \ref{topsm}.

As in Section 2, we think of $K^{\mathbf{0}}_{\lambda}(X)$ as homotopy classes of \emph{lifts} of the classifying map $\lambda:X\rightarrow S^{2k+1}$ to a map $F:X\rightarrow Y_{2k+1}$.  Similarly, $K^{\mathbf{1}}_{\lambda}(X)$ can be described as homotopy classes $G:X\rightarrow L_{2k+1}$ of $\lambda$.  The relative twisted Chern character for CW pairs with cohomotopy twist may now be defined as follows.

\begin{definition}
	Let $(X,A,H)$ be a finite CW pair with cohomotopy twist.  Let $\alpha$ denote the map from Theorem \ref{iso2} sending manifold de Rham forms to singular de Rham forms.  The \textbf{even twisted Chern character} is the map $Ch^{\mathbf{0}}_{H}:K^{\mathbf{0}}_{[H]}(X,A)\rightarrow H^{\mathbf{0}}_{H}(X,A)$ defined by
	\[
	Ch^{\mathbf{0}}_{H}([F]):=[F^{*}\overline{\alpha(Ch^{\mathbf{0}}_{2k+1})}],
	\]
	where $F:X\rightarrow Y_{2k+1}$ is a lift of $\tilde{H}:X\rightarrow S^{2k+1}$ which coincides with the identity section outside of some compact set in $X-A$, and where $\overline{\alpha(Ch^{\mathbf{0}}_{2k+1})}$ is any extension of $\alpha(Ch^{\mathbf{0}}_{2k+1})$ to a closed singular form for $Y_{2k+1}$ with its continuous diffeology as in Proposition \ref{topsm}.  Similarly, the \textbf{odd twisted Chern character} is the map $Ch^{\mathbf{1}}_{H}:K^{\mathbf{1}}_{[H]}(X,A)\rightarrow H^{\mathbf{1}}_{H}(X,A)$ defined by
	\[
	Ch^{\mathbf{1}}_{H}([G]):=[G^{*}\overline{\alpha(Ch^{\mathbf{1}}_{2k+1})}],
	\]
	where $G:X\rightarrow L_{2k+1}$ is a lift of $\tilde{H}:X\rightarrow S^{2k+1}$ which coincides with the identity section outside of some compact set in $X-A$, and where $\overline{\alpha(Ch^{\mathbf{1}}_{2k+1})}$ is any extension of $\alpha(Ch^{\mathbf{1}}_{2k+1})$ to a closed singular form for $L_{2k+1}$ with its continuous diffeology as in Proposition \ref{topsm}.
\end{definition}

Let us remark that the odd and even twisted Chern characters do indeed take values in the claimed relative cohomology groups.  Indeed, if $\sigma:\Delta^{l}\rightarrow A$ is any singular simplex and $F:X\rightarrow Y_{2k+1}$ is any lift of $\tilde{H}$ which coincides with the identity outside of some compact set in $X-A$, then $F\circ\sigma = \id_{Y_{2k+1}}\circ \tilde{H}\circ\sigma$.  Since $\id_{Y_{2k+1}}^{*}Ch^{\mathbf{0}}_{2k+1}$ is the zero form by Lemma \ref{pullbackid}, for any singular simplex $\sigma$ in $A$ we apply Proposition \ref{compat} to see that
\begin{align*}
\big(F^{*}\overline{\alpha(Ch^{\mathbf{0}}_{2k+1})}\big)(\sigma) =& \overline{\alpha(Ch^{\mathbf{0}}_{2k+1})}(\id_{Y_{2k+1}}\circ \tilde{H}\circ\sigma) = \big(\id_{Y_{2k+1}}^{*}\overline{\alpha(Ch^{\mathbf{0}}_{2k+1})}\big)(\tilde{H}\circ\sigma)\\ =& \overline{(\id_{Y_{2k+1}}^{*}\alpha(Ch^{\mathbf{0}}_{2k+1}))}(\tilde{H}\circ\sigma) = \overline{\alpha(\id_{Y_{2k+1}}^{*}Ch^{\mathbf{0}}_{2k+1})}(\tilde{H}\circ\sigma) = 0,
\end{align*}
so that $F^{*}Ch^{\mathbf{0}}_{2k+1}\in\Omega^{*}_{sing}(X,A)$.  Here the third equality follows from the fact that $\id_{Y_{2k+1}}^{*}\overline{\alpha(Ch^{\mathbf{0}}_{2k+1})}$ is an extension of $\id_{Y_{2k+1}}^{*}\alpha(Ch^{\mathbf{0}}_{2k+1})$.  Similar arguments apply for $Ch^{\mathbf{1}}_{H}$.  The Chern character can now be seen to be a natural transformation of generalised cohomology theories.

\begin{theorem}\label{chnat}
	The twisted Chern character is a natural transformation of generalised cohomology theories from the category of finite CW pairs with cohomotopy twist to the category of abelian groups.  That is, given a finite CW pair with cohomotopy twist $(X,A,H)$, for each $*=\mathbf{0},\mathbf{1}$ the diagram
	\begin{center}
		\begin{tikzcd}
		K^{*}_{[H]}(X,A) \ar[d,"Ch^{*}_{H}"] \ar[r,"\ext_{*}"] & K^{*}_{[H]}(X) \ar[d,"Ch^{*}_{H}"] \ar[r,"\rest_{*}"] & K^{*}_{[H]}(A) \ar[d,"Ch^{*}_{H}"] \ar[r,"\partial_{*}"] &K^{*+\mathbf{1}}_{[H]}(X,A) \ar[d,"Ch^{*+\mathbf{1}}_{H}"]\\ H^{*}_{H}(X,A) \ar[r,"\iota_{*}"] & H^{*}_{H}(X) \ar[r,"\pi_{*}"] & H^{*}_{H}(X,A) \ar[r,"D_{*}"] & H^{*+\mathbf{1}}_{H}(X,A)
		\end{tikzcd}
	\end{center}
	commutes.
\end{theorem}

\begin{proof}
	To show commutativity of the first two squares, it suffices to show that if $f:(X,A,H_{X})\rightarrow (Y,B,H_{Y})$ is a morphism of finite CW pairs with cohomotopy twist, so that $\tilde{H}_{Y}\circ f = \tilde{H}_{X}$, then the square
	\begin{center}
		\begin{tikzcd}
		K^{*}_{[H_{Y}]}(Y,B) \ar[r,"f^{*}"] \ar[d,"Ch_{H_{Y}}^{*}"] & K^{*}_{[H_{X}]}(X,A) \ar[d,"Ch_{H_{X}}^{*}"] \\ H^{*}_{H_{Y}}(Y,B) \ar[r,"f^{*}"] & H^{*}_{H_{X}}(X,A)
		\end{tikzcd}
	\end{center}
	commutes.  Suppose that $*=\mathbf{0}$.  The $*=\mathbf{1}$ case follows by essentially the same argument.  Then given a lift $F:Y\rightarrow Y_{2k+1}$ of $\tilde{H}_{Y}$ which coincides with the identity section outside of some compact subset of $Y-B$, we compute
	\begin{align*}
	Ch^{\mathbf{0}}(f^{*}[F]) =& Ch^{\mathbf{0}}([F\circ f]) = [(F\circ f)^{*}\overline{\alpha(Ch^{\mathbf{0}}_{2k+1})}] = [f^{*}(F^{*}\overline{\alpha(Ch^{\mathbf{0}}_{2k+1})})]\\ =& f^{*}[F^{*}\overline{\alpha(Ch^{\mathbf{0}}_{2k+1})}] = f^{*}Ch^{\mathbf{0}}([F])
	\end{align*}
	giving the claimed commutativity.  Commutativity of the final square follows from a retraction of the open suspension $c(X,A)-X$ of $A$ onto $A$ (see our topological definition of the connecting homomorphism $\partial$ given in Equation \eqref{connecting}), together with the fact that the odd Chern form on the identity-based loop space of the Fredholm operators transgresses the even Chern form on the Fredholm operators.
\end{proof}

Taking a finite CW complex with cohomotopy twist $(X,H)$ now, and replacing $X$ and $A$ in the statement of Theorem \ref{chnat} with skeleta $X^{p}$ and $X^{p-1}$ respectively, we have that the Chern character induces a morphism of spectral sequences from the $\{E_{r}(K)\}$ of Theorem \ref{AH} to the $\{E_{r}(H)\}$ of Theorem \ref{specseqcohom}.  In particular, on the first page, a choice of trivialisation of the twist over $X^{p}-X^{p-1}$ (a disjoint union of contractible spaces) identifies the twisted Chern character with the ordinary Chern character, so that the square
\begin{center}
	\begin{tikzcd}
	K^{*}_{[H]}(X^{p},X^{p-1}) \ar[r,"Ch_{H}^{*}"] \ar[d,"\cong"] & H^{*}_{H}(X^{p},X^{p-1}) \ar[d,"\cong"] \\ K^{*}(X^{p},X^{p-1}) \ar[r,"Ch^{*}"] & H^{*}(X^{p},X^{p-1};\RR)
	\end{tikzcd}
\end{center}
commutes.  The expected theorem then follows.

\begin{theorem}\label{chiso}
	Given any finite CW complex with cohomotopy twist $(X,H)$, the twisted Chern character induces an isomorphism of the real-ised Atiyah-Hirzebruch spectral sequence $\{E_{r}(K)\otimes\RR\}$ onto the spectral sequence $\{E_{r}(H)\}$ of Theorem \ref{specseqcohom} computing the higher twisted cohomology.  Consequently the Chern character defines a real isomorphism $Ch^{*}_{H}:K^{*}_{[H]}(X)\otimes\RR\rightarrow H^{*}_{H}(X)$.\qed
\end{theorem}

\begin{remark}\normalfont
	Note that the spectral sequence we have used in Theorem \ref{specseqcohom} to compute higher twisted cohomology \emph{distinct} from that used by Atiyah and Segal in \cite{AS1}.  Specifically, while our spectral sequence has been constructed from the filtration of the underlying space by its skeleta, that of Atiyah and Segal used a filtration by degree of differential form (we refer the reader to \cite{LLW} for a detailed construction of the spectral sequence filtered by form degree in the higher twisted case).  Our reason for choosing a different spectral sequence is that it is unclear even in the case of the 3-twist considered by Atiyah and Segal why the twisted Chern character induces the claimed morphism of spectral sequences when one filters the twisted de Rham complex by degree instead of by skeleta.  Atiyah and Segal claim that this follows ``Because the twisted Chern character is functorial" \cite[p. 24]{AS1} - however, as is clear from our presentation, naturality of the Chern character is with respect to CW pairs and not, a-priori, with respect to degree of differential forms.
	
	In particular, we have not been able to prove that the higher differentials of the real-ised Atiyah-Hirzebruch spectral sequence are Massey products with the twisting form.  This computation appears to require the filtration by degree of differential forms (see \cite[Theorem 1.2]{LLW} for a detailed exposition).  Thus to obtain the higher differentials, one must find a way of relating the filtration of higher twisted cohomology by skeleta with the filtration by degree of differential forms.  We leave this as an open problem.
\end{remark}

\end{document}